\documentclass{article}
\usepackage{amsmath, amssymb, amsthm}

\usepackage[dvips]{graphicx}
\usepackage{amsrefs}

\newtheorem{theo}{Theorem}

\newtheorem{lemm}[theo]{Lemma}
\newtheorem{coro}[theo]{Corollary}
\newtheorem{prop}[theo]{Proposition}

\newtheorem*{them}{Theorem}

\title{On cobrackets on the Wilson loops associated with flat $\mathrm{GL}(1, \mathbb{R})$-bundles over surfaces}
\author{MOEKA NOBUTA}
\date{}

\begin{document}
\maketitle
\begin{abstract}
Let $S$ be a closed connected oriented surface of genus $g>0$. We study a Poisson subalgebra $W_1(g)$ of $C^{\infty}(\mathrm{Hom}(\pi_1(S), \mathrm{GL}(1, \mathbb{R}))/\mathrm{GL}(1, \mathbb{R}))$, the smooth functions on the moduli space of flat $\mathrm{GL}(1, \mathbb{R})$-bundles over $S$. There is a surjective Lie algebra homomorphism from the Goldman Lie algebra onto $W_1(g)$.
We classify all cobrackets on $W_1(g)$ up to coboundary, that is, we compute $H^1(W_1(g), W_1(g)\wedge W_1(g))\cong \mathrm{Hom}(\mathbb{Z}^{2g}, \mathbb{R}).$
As a result, there is no cohomology class corresponding to the Turaev cobracket on $W_1(g)$.
\end{abstract}
\begin{section}{Introduction}

Atiyah and Bott \cite{AB83} constructed a symplectic structure on $\mathrm{Hom}(\pi_1(S), G)/G$, the moduli space of flat $G$-bundles over a closed connected oriented surface $S$, where $G$ is a compact quadratic Lie group, that is, a compact Lie group together with a non-degenerate $\mathrm{Ad}(G)$-invariant symmetric bilinear form on the corresponding Lie algebra.
Goldman \cite{Gol84} extended this construction to any non-compact quadratic Lie group $G$.
By using the symplectic structure, we define the structure of a Poisson algebra on the smooth functions $C^{\infty}(\mathrm{Hom}(\pi_1(S), G)/G)$.
Let $\hat\pi$ be the set of free homotopy classes of free loops on $S$ and $\mathbb{Z}[\hat\pi]$ the free $\mathbb{Z}$-module generated by $\hat\pi$.
Goldman \cite{Gol86} constructed a Lie algebra structure on $\mathbb{Z}[\hat\pi]$ and established formulas which connect the Lie algebra $\mathbb{Z}[\hat\pi]$ and the Poisson algebra $C^{\infty}(\mathrm{Hom}(\pi_1(S), G)/G)$. 

Since the constant loop $1$ is in the center of $\mathbb{Z}[\hat\pi]$, the quotient $\mathbb{Z}[\hat\pi]_0:=\mathbb{Z}[\hat\pi]/\mathbb{Z}1$ has a Lie algebra structure.
Turaev \cite{Tur91} constructed a cobracket on $\mathbb{Z}[\hat\pi]_0$ which is compatible with the Lie bracket.
As an application of Turaev cobracket, Chas and Krongold \cite{CK15} gives a characterization of simple closed curves. 
Alekseev, Kawazumi, Kuno and Naef \cite{AKKN} introduced higher genus Kashiwara-Vergne problems in view of the formality problem of the Turaev cobracket.

In the case $G=\mathrm{GL}(n, \mathbb{R})$, we consider a non-degenerate $\mathrm{Ad}(G)$-invariant symmetric bilinear form
\begin{align*}
\mathrm{Tr}:\mathfrak{gl}(n, \mathbb{R})\times\mathfrak{gl}(n, \mathbb{R})\rightarrow\mathbb{R},
(A, B)\mapsto\mathrm{Tr}(AB).
\end{align*}
Let $W_n(g)$ be the Lie subalgebra of $C^{\infty}(\mathrm{Hom}(\pi_1(S), \mathrm{GL}(n, \mathbb{R}))/\mathrm{GL}(n, \mathbb{R}))$ generated by all Wilson loops, where $g$ is the genus of the surface $S$.
For $[\gamma]\in\hat\pi$ represented by $\gamma\in\pi_1(S)$, the Wilson loop associated with $[\gamma]$ is given by
\begin{align*}
\omega_{[\gamma]}:\mathrm{Hom}(\pi_1(S), \mathrm{GL}(n, \mathbb{R}))\rightarrow\mathbb{R},\
\rho\mapsto\omega_{[\gamma]}:=\mathrm{Tr}(\rho(\gamma)).
\end{align*}
This induces a Lie algebra homomorphism \cite{Gol86}
\begin{align*}
\mathbb{Z}[\hat\pi]_0\subset\mathbb{Z}[\hat\pi]\rightarrow W_n(g)\subset C^{\infty}(\mathrm{Hom}(\pi_1(S), \mathrm{GL}(n, \mathbb{R}))/\mathrm{GL}(n, \mathbb{R})),
[\gamma]\mapsto\omega_{[\gamma]}.
\end{align*}
So it is natural to ask whether there exists a cohomology class corresponding to the Turaev cobeacket on $W_n(g)$ and $C^{\infty}(\mathrm{Hom}(\pi_1(S), \mathrm{GL}(n, \mathbb{R}))/\mathrm{GL}(n, \mathbb{R}))$.

In this paper we discuss about cobrackets on $W_1(g)$ which are compatible with the bracket and classify them up to coboundaries.
When $n=1$, $W_1(g)$ is identified with polynomials $\mathbb{R}[x_1, y_1, \cdots, x_g, y_g, x_1^{-1}, y_1^{-1}, \cdots, x_g^{-1}, y_g^{-1}]$, where $(x_1, y_1, \cdots, x_g, y_g)$ is a symplectic generators of $\pi_1(S, *)$.
So $W_1(g)$ is easier than other $W_n(g)$'s but there are non-trivial cobrackets on $W_1(g)$.

We obtain following isomorphism.
\begin{them}
\begin{align*}
\mathrm{Hom}(\mathbb{Z}^{2g}, \mathbb{R})\cong H^1(W_1(g), W_1(g)\wedge W_1(g)),\
k\mapsto [\Delta_k],
\end{align*}
where
$\Delta_k:W_1(g)\rightarrow W_1(g)\wedge W_1(g)$ is given by
$x_1^{a_1}y_1^{b_1}\cdots x_g^{a_g}y_g^{b_g}\in\pi_1(S)^{\mathrm{Ab}}\mapsto k(a_1, b_1, \cdots, a_g, b_g)x_1^{a_1}y_1^{b_1}\cdots x_g^{a_g}y_g^{b_g}\wedge 1$.
\end{them}
$\Delta_k$ is a cobracket compatible with the bracket. 
If we consider a compact surface with non-empty boundary instead of the closed surface $S$, then there is a framed Turaev cobracket $\delta^f$ on the compact surface associated with a framing $f$ of the tangent bundle \cite{AKKN}.
Since $\delta^{f+\chi}(\alpha)=\delta^f(\alpha)+\chi(\alpha)1\wedge\alpha$ for any first cohomology class $\chi$ of the surface, $\Delta_k$ corresponds to the change of framing of Turaev cobracket. But there is no cohomology class corresponding to the framed Turaev cobrackets on $W_1(g)$ which represents a non-trivial cohomology class in $H^1(\mathbb{Z}[\hat\pi]_0, \mathbb{Z}[\hat\pi]_0\wedge\mathbb{Z}[\hat\pi]_0)$.
We show the non-triviality in Section 3.

\subsection*{Acknowledgements}
The author would like to thank her advisor, Nariya Kawazumi for his dedicated support and continuous encouragement.

\end{section}

\begin{section}{Cobrackets on $W_1(g)$}
As was pointed out by Drinfel'd \cite{Dri}, the compatible condition for cobrackets is equivalent to the cocycle condition for $1$-cochains.
So we may consider cobrackets as $1$-cocycles and are led to compute the first cohomology group.
We show the map
\[
\mathrm{Hom}(\mathbb{Z}^{2g}, \mathbb{R})\rightarrow H^1(W_1(g), W_1(g)\wedge W_1(g)), k\mapsto [\Delta_k]
\]
is injective in section \ref{section1} and surjective in section \ref{section2}. 

Let $(\mathfrak{g}, [\ ,\ ])$ be a Lie algebra  over the field of real numbers $\mathbb{R}$, and $M$ a left $\mathfrak{g}$-module.
$Z^1(\mathfrak{g}, M)$ denotes the set of all linear functions $f:\mathfrak{g}\rightarrow M$ satisfying
$\gamma_0\cdot f(\gamma_1)-\gamma_1\cdot f(\gamma_0)-f([\gamma_0, \gamma_1])=0$
and $B^1(\mathfrak{g}, M)$ the image of $M$ under $d:M\rightarrow Z^1(\mathfrak{g}, M)$, where $d(m)(\gamma)=\gamma\cdot m$ for $m\in M$ and $\gamma\in\mathfrak{g}$.
The first cohomology group of the Lie algebra $\mathfrak{g}$ with coefficients in $M$ is defined by 
$H^1(\mathfrak{g}, M)=Z^1(\mathfrak{g}, M)/B^1(\mathfrak{g}, M)$.

The products $W_1(g)\wedge W_1(g)$ and $W_1(g)\otimes W_1(g)$ are $W_1(g)$-modules given by
$Z_0\cdot (Z_1\wedge Z_2)= \{Z_0, Z_1\}\wedge Z_2+Z_1\wedge \{Z_0, Z_2\}$ and
$Z_0\cdot (Z_1\otimes Z_2)= \{Z_0, Z_1\}\otimes Z_2+Z_1\otimes \{Z_0, Z_2\}$.

If $\Delta:W_1(g)\rightarrow W_1(g)\wedge W_1(g)$ is compatible with the bracket $\{\ ,\ \}$, then $\Delta\in Z^1(W_1(g), W_1(g)\wedge W_1(g))$, that is,
\begin{align*}
&Z_0\cdot \Delta(Z_1)-Z_1\cdot \Delta(Z_0)-\Delta(\{Z_0, Z_1\})\\
&=\{Z_0, \Delta(Z_1)\}+\{\Delta(Z_0), Z_1\}-\Delta(\{Z_0, Z_1\})=0,
\end{align*}
where
$\{Z, u\wedge v\}=\{Z, u\}\wedge v+u\wedge\{Z, v\}$ and $\{u\wedge v, Z\}=-\{Z, u\wedge v\}$ for $u, v\in W_1(g)$. 
We denote by $[\Delta]$ the cohomology class of $\Delta$.

If $n=1$, the Wilson loop associated with $[\gamma]\in\hat\pi$ can be identified with the image of $\gamma$ under the natural surjection $\pi_1(S)\rightarrow\pi_1(S)^{\mathrm{Ab}}$.
Fix symplectic generators $(x_1, y_1, \cdots ,x_g, y_g)$ of $\pi_1(S)=\pi_1(S, *)$ as shown in Figure \ref{fig:closed_surface}.

\begin{figure}
  \begin{center}
    \includegraphics[width=8cm]{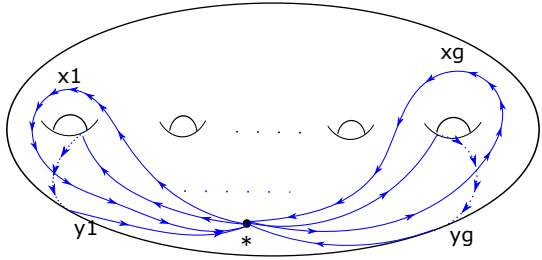}
    \caption{A surface $S$ with symplectic generators of $\pi_1(S)$.}
    \label{fig:closed_surface}
 \end{center}
\end{figure}

Then we have a $\mathbb{R}$-module isomorphism $W_1(g)\cong\mathbb{R}\pi_1(S)^{\mathrm{Ab}}=\mathbb{R}[x_1, y_1, \cdots, x_g, y_g, x_1^{-1}, y_1^{-1}, \cdots, x_g^{-1}, y_g^{-1}]$.

The bracket on $W_1(g)$ is defined by
$\{Z_1, Z_2\}=i(Z_1, Z_2)Z_1Z_2$
for $Z_1$ and $Z_2\in\pi_1(S)^{\mathrm{Ab}}$, where $i$ is the algebraic intersection form, that is,\\ 
$i(x_1^{a_1}y_1^{y_1}\cdots x_g^{a_g}y_g^{b_g}, x_1^{a'_1}y_1^{b'_1}\cdots x_g^{a'_g}y_g^{b'_g})=a_1b'_1-a'_1b_1+\cdots +a_gb'_g-a'_gb_g.$

\begin{subsection}{Compatibility on $W_1(g)$}\label{section1}
In this section we show the map $\mathrm{Hom}(\mathbb{Z}^{2g}, \mathbb{R})\rightarrow H^1(W_1(g), W_1(g)\wedge W_1(g))$ is injective and for $\Delta\in Z^1(W_1(g), W_1(g)\wedge W_1(g))$, there exist some homomorphism $k\in\mathrm{Hom}(\mathbb{Z}^{2g}, \mathbb{R})$ and some element $\alpha\in W_1(g)\wedge 1$ such that $(\Delta-(\Delta_k+d\alpha))(Z)\in W_1(g)'\wedge W_1(g)'$ for all $Z\in W_1(g)$ where 
\[
W_1(g)':=\mathbb{R}\{\omega_{\gamma}\ |\ \gamma\neq 1\in\pi_1(S)^{\mathrm{Ab}}\}\subset W_1(g),
\] 
the submodule generated by all Wilson loops except for the Wilson loop associated with trivial loop.
We have decompositions of $W_1(g)$-modules
$W_1(g)\wedge W_1(g) =(W_1(g)'\wedge 1)\oplus(W_1(g)'\wedge W_1(g)') $ and $W_1(g)\otimes W_1(g)=\mathbb{R}(1\otimes 1)\oplus(W_1(g)'\otimes 1)\oplus(1\otimes W_1(g)')\oplus(W_1(g)'\otimes W_1(g)')$.

For a map $k:\mathbb{Z}^{2g}\rightarrow\mathbb{R}$, define a linear map
$\Delta_k:W_1(g)\rightarrow W_1(g)\wedge W_1(g)$ by
\begin{align*}
\Delta_k(x_1^{a_1}y_1^{b_1}\cdots x_g^{a_g}y_g^{b_g})=k(a_1, b_1, \cdots, a_g, b_g)x_1^{a_1}y_1^{b_1}\cdots x_g^{a_g}y_g^{b_g}\wedge 1.
\end{align*}
If $\Delta_k$ is compatible with the bracket $\{\ ,\ \}$, then $\Delta_k$ can be seen as a $1$-cocycle of $W_1(g)$ with values in $W_1(g)\wedge W_1(g)$.

\begin{lemm}\label{lem:k0hom}
Suppose a map $k:\mathbb{Z}^{2g}\rightarrow \mathbb{R}$ satisfies the condition $k(0, \cdots , 0)=0$ and
$k(a_1+a'_1, b_1+b'_1, \cdots , a_g+a'_g, b_g+b'_g)=k(a_1, b_1, \cdots , a_g, b_g)+k(a'_1, b'_1, \cdots , a'_g, b'_g)$
for $a_1b'_1-b_1a'_1+\cdots +a_g b'_g-b_g a'_g \neq 0$. Then $k$ is a homomorphism.
\end{lemm}
\begin{proof}

We will show the statement by induction on $g>0$.
In the case $g=1$, suppose $k(0, 0)=0$ and $k(a+a', b+b')=k(a, b)+k(a', b')$ for $ab'-ba'\neq 0$.
For $a>0, b>0$, we have
\begin{eqnarray*}
k(a, b)&=& k(a, b-1)+k(0, 1)= \cdots  = k(a, 1)+(b-1)k(0, 1)\\
&=& k(a-1, 1)+k(1, 0)+(b-1)k(0, 1)\\
&=& \cdots =k(0, 1)+ak(1, 0)+(b-1)k(0, 1)= ak(1, 0)+bk(0, 1).
\end{eqnarray*}
Since $k(a, 0)+k(0, 1)=k(a, 1)=ak(1, 0)+k(0, 1)$, we obtain $k(a, 0)=ak(1, 0)$.
Similarly, we have $k(0, b)=bk(0, 1)$.
Since $k(a, b)+k(a, -b)=k(2a, 0)$,
\begin{eqnarray*}
k(a, -b)&=&k(2a, 0)-k(a, b)=2ak(1, 0)-(ak(1, 0)+bk(0, 1))\\
&=&ak(1, 0)-bk(0, 1).
\end{eqnarray*} 
Similarly we have
$k(-a, b)=-ak(1, 0)+bk(0, 1)$, $k(-a, 0)=-ak(1, 0)$, 
$k(0, -b)=-bk(0, 1)$, and $k(-a, -b)= -ak(1, 0)-bk(0, 1)$.
By the condition $k(0, 0)=0$, $k$ is a homomorphism.

For $g\geq 2$, suppose the above statement holds true for $g-1$, $k(0, \cdots, 0)=0$ and 
$k(a_1+a'_1, b_1+b'_1, \cdots , a_g+a'_g, b_g+b'_g)=k(a_1, b_1, \cdots , a_g, b_g)+k(a'_1, b'_1, \cdots , a'_g, b'_g)$
for $a_1b'_1-b_1a'_1+\cdots +a_g b'_g-b_g a'_g \neq 0$.
Set
\begin{align*}k_{g-1}(a_1, b_1, \cdots , a_{g-1}, b_{g-1})&:=k(a_1, b_1, \cdots , a_{g-1}, b_{g-1}, 0, 0),\\
k_1(a_g, b_g)&:=k(0, \cdots , 0, a_g, b_g).
\end{align*}
From the inductive assumption, $k_{g-1}$ and $k_1$ are homomorphisms.
So if $(a_1, b_1, \cdots , a_{g-1}, b_{g-1})=(0, \cdots , 0)$ or $(a_g, b_g)=(0, 0)$, we have 
$k(a_1, b_1, \cdots , a_g, b_g)=a_1k(1, 0, \cdots , 0)+b_1k(0, 1, 0, \cdots , 0)+\cdots +a_gk(0, \cdots , 0, 1, 0)+b_gk(0, \cdots , 0, 1)$.
We can assume $(a_1, b_1, \cdots , a_{g-1}, b_{g-1})\neq(0, \cdots , 0)$ and $(a_g, b_g)\neq(0, 0)$.
Suppose $a_{g-1}\neq 0$ and $a_g\neq 0$.
By the assumption, we have
\begin{eqnarray}
&&k(a_1, b_1, \cdots , a_g, b_g) = k(a_1, b_1, \cdots , a_{g-1}, b_{g-1}-1, 0, 0)+k(0, \cdots , 0, 1, a_g, b_g) \nonumber\\
&=& k(a_1, b_1, \cdots, a_{g-1}, b_{g-1}-1, 0, 0)+k(0, \cdots , 0, 1, 0, 1)+k(0, \cdots , 0, a_g, b_g-1)\nonumber\\
&=& k_{g-1}(a_1, b_1, \cdots , a_{g-1}, b_{g-1}-1)+k(0, \cdots , 0, 1, 0, 1)+k_1(a_g, b_g-1).\label{eq:kg-1k1}
\end{eqnarray}
Take $(c_1, \cdots , c_g, d_1, \cdots , d_g, c'_1, \cdots , c'_g, d'_1, \cdots , d'_g)$ satisfying $c_1d'_1-d_1c'_1+\cdots +c_gd'_g-d_gc'_g \neq 0$,
$c_{g-1} \neq 0$, $c_g \neq 0$, $c'_{g-1} \neq 0$, $c'_g \neq 0$,
$c_{g-1}+c'_{g-1} \neq 0$, and $c_g+c'_g \neq 0$.
For example we can take $(c_1, d_1, \cdots, c_g, d_g)=(0, \cdots, 0, 1, 0, 1, 0)$ and $(c'_1, d'_1, \cdots, c'_g, d'_g)=(0, \cdots, 0, 1, 0, 1, 1)$ as in Figure\ref{fig:example1}.
\begin{figure}[htbp]
  \begin{center}
    \includegraphics[width=10cm,clip]{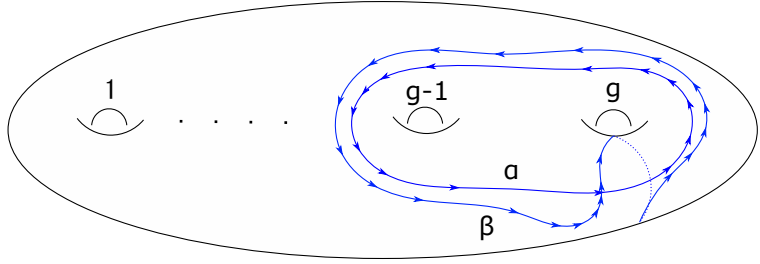}
    \caption{$\alpha=x_{g-1}x_g, \beta =x_{g-1}x_gy_g$.}
    \label{fig:example1}
  \end{center}
\end{figure}

Applying (\ref{eq:kg-1k1}) to $(c_1, \cdots , c_g, d_1, \cdots , d_g, c'_1, \cdots , c'_g, d'_1, \cdots , d'_g)$, we obtain
\begin{eqnarray*}
&&k_{g-1}(c_1+c'_1, d_1+d'_1, \cdots , c_{g-1}+c'_{g-1}, d_{g-1}+d'_{g-1}-1)\\&&+k(0, \cdots , 0, 1, 0, 1)+k_1(c_g+c'_g, d_g+d'_g-1) \\
&=& k(c_1+c'_1, d_1+d'_1, \cdots , c_g+c'_g, d_g+d'_g)\\
&=& k(c_1, d_1, \cdots , c_g, d_g)+k(c'_1, d'_1, \cdots , c'_g, d'_g)\\
&=& k_{g-1}(c_1, d_1, \cdots , c_{g-1}, d_{g-1}-1)+k(0, \cdots , 0, 1, 0, 1)+k_1(c_g, d_g-1)\\
&&+k_{g-1}(c'_1, d'_1, \cdots , c'_{g-1}, d'_{g-1}-1)+k(0, \cdots , 0, 1, 0, 1)+k_1(c'_g, d'_g-1).
\end{eqnarray*}
Since $k_{g-1}$ and $k_1$ are homomorphisms, we have
$k(0, \cdots , 0, 1, 0, 1)=k(0, \cdots , 0, 1, 0, 0)+k(0, \cdots , 0, 1)$.
Therefore if $a_{g-1}\neq 0$ and $a_g\neq 0$, then
$k(a_1, b_1, \cdots , a_g, b_g)=a_1k(1, 0, \cdots , 0)+b_1k(0, 1, 0, \cdots , 0)+\cdots +a_gk(0, \cdots , 0, 1, 0)+b_gk(0, \cdots , 0, 1)$.
Similar argument holds for other cases.
\end{proof}

\begin{theo}\label{thm:hom}
For any map $k:\mathbb{Z}^{2g}\rightarrow \mathbb{R}$,
$\Delta_k\in Z^1(W_1(g), W_1(g)\wedge W_1(g))$ if and only if there exists a homomorphism $k':\mathbb{Z}^{2g}\rightarrow \mathbb{R}$ such that $k'|_{\mathbb{Z}^{2g}\backslash(0, \cdots, 0)}=k|_{\mathbb{Z}^{2g}\backslash(0, \cdots, 0)}$.
\end{theo}
\begin{proof}
The Goldman Lie bracket 
$[\ ,\ ]$ on the Goldman Lie algebra $\mathbb{Z}[\hat\pi]$ is defined as follows.
For $\alpha=[a]$ and $\beta=[b]\in\hat\pi$,
$[\alpha, \beta]=\sum_{p\in a\cap b}\epsilon(p; a, b)[a_pb_p]$,
where $a$ and $b$ are generic immersions, $a_pb_p$ denote the product of $a$ and $b$ as based loops in $\pi_1(S, p)$ and $\epsilon(p; a, b)=1$ if the orientation given by the pair of vectors $\{a'(s), b'(t)\}$ agrees with the orientation of $S$ where $a(s)=b(t)=p$, $\epsilon(p; a, b)=-1$ otherwise.

Let $k$ be a map from $\mathbb{Z}^{2g}$ to $\mathbb{R}$.
Since the map $\mathbb{Z}[\hat\pi]\rightarrow W_1(g), [\gamma]\mapsto \omega_{[\gamma]}$ is a Lie algebra homomorphism \cite{Gol86}, we have
\begin{eqnarray*}
\Delta_k \{\omega_{\alpha}, \omega_{\beta}\}
&=&\Delta_k(\sum_{p\in a\cap b}\epsilon(p; \alpha, \beta)\omega_{[a_p b_p]})= \sum_{p\in a\cap b} k([a_p b_p])\epsilon(p; \alpha, \beta)\omega_{[a_p b_p]} \wedge 1\\
&=& k([a_p b_p])\{\omega_{\alpha}, \omega_{\beta}\}\wedge 1,
\end{eqnarray*}
\begin{eqnarray*}
\{\Delta_k(\omega_{\alpha}), \omega_{\beta}\}+\{\omega_{\alpha}, \Delta_k(\omega_{\beta})\} &=& k(\alpha)\{\omega_{\alpha}\wedge 1, \omega_{\beta}\}+k(\beta)\{\omega_{\alpha}, \omega_{\beta}\wedge 1\}\\
&=& (k(\alpha)+k(\beta))\{\omega_{\alpha}, \omega_{\beta}\}\wedge 1,
\end{eqnarray*}
where $k(\gamma):=k(a_1, b_1, \cdots a_g, b_g)$ for $\gamma\in\hat\pi \mapsto x_1^{a_1}y_1^{b_1}\cdots x_g^{1_g}y_g^{b_g}\in\pi_1(S)^\mathrm{Ab}$.
So if there exists a homomorphism $k':\mathbb{Z}^{2g}\rightarrow\mathbb{R}$ such that $k'|_{\mathbb{Z}^{2g}\backslash(0, \cdots, 0)}=k|_{\mathbb{Z}^{2g}\backslash(0, \cdots, 0)}$, then $\Delta_k$ is compatible with the bracket.

Suppose $\Delta_k$ is compatible with the bracket for some $k:\mathbb{Z}^{2g}\rightarrow\mathbb{R}$.
Then $k$ satisfies the latter condition in Lemma \ref{lem:k0hom}.
Hence we have a homomorphism $k':\mathbb{Z}^{2g}\rightarrow\mathbb{R}$ such that $k'|_{\mathbb{Z}^{2g}\backslash (0, \cdots, 0)}=k|_{\mathbb{Z}^{2g}\backslash (0, \cdots, 0)}$.
\end{proof}

\begin{theo}\label{thm:wedge1}
\begin{enumerate}
\item The map $\mathrm{Hom}(\mathbb{Z}^{2g}, \mathbb{R})\rightarrow H^1(W_1(g), W_1(g)\wedge 1), k\mapsto [\Delta_k]$ is bijective.
\item The map $H^1(W_1(g), W_1(g)\wedge 1)\rightarrow H^1(W_1(g), W_1(g)\wedge W_1(g))$ induced by the inclusion homomorphism $1\rightarrow W_1(g)$ is injective.
\end{enumerate}
\end{theo}
\begin{proof}
We begin by showing the map
$\mathrm{Hom}(\mathbb{Z}^{2g}, \mathbb{R})\rightarrow H^1(W_1(g), W_1(g)\wedge 1)$
is injective.
Suppose $\Delta_k$ is a coboundary, that is,
$\Delta_k=\sum_{u\in\pi_1 (S)^{\mathrm{Ab}}}C_{u} d(u\wedge 1)$, 
for some $C_{u}\in\mathbb{R}$.
Therefore
$k(Z)Z\wedge 1=\Delta_k(Z)=\sum_{u\in\pi_1 (S)^{\mathrm{Ab}}}C_{u}\{Z, u\}\wedge 1$
for every $Z\in\pi_1(S)^{\mathrm{Ab}}$.
We look at the coefficient of $Z\wedge 1$ in the right hand side.
Since, for any $u\in\pi_1(S)^{\mathrm{Ab}}$, $\{Z, u\}=i(Z, u)Zu\notin\mathbb{R}Z-0$ where $\mathbb{R}Z\subset W_1(g)$ is the $\mathbb{R}$-submodule generated by $Z$,
we obtain
$k(Z)=0$
for all $Z\neq 1\in\pi_1(S)^{\mathrm{Ab}}$.
Hence $[\Delta_k]=0\in H^1(W_1(g), W_1(g)\wedge 1)$ if and only if $k=0$.
Therefore the map $\mathrm{Hom}(\mathbb{Z}^{2g}, \mathbb{R})\rightarrow H^1(W_1(g), W_1(g)\wedge 1)$ is injective.

Next we will show the map
$\mathrm{Hom}(\mathbb{Z}^{2g}, \mathbb{R})\rightarrow H^1(W_1(g), W_1(g)\wedge 1)$
is surjective.
Let $[\Delta]\in H^1(W_1(g), W_1(g)\wedge 1)$ be represented by
$\Delta(Z)=\sum_{u\in\pi_1 (S)^{\mathrm{Ab}}}C_{u}^Z u\wedge 1$
for $Z\in\pi_1(S)^{\mathrm{Ab}}$.
Since $\Delta$ is compatible with the bracket, we have
\begin{eqnarray*}
(a_1b'_1-b_1a'_1+\cdots +a_gb'_g-b_ga'_g)\Delta(ZZ')
=\{\Delta(Z), Z'\}+\{Z, \Delta(Z')\},
\end{eqnarray*}
for $Z=x_1^{a_1}y_1^{b_1}\cdots x_g^{a_g}y_g^{b_g}$ and $Z'=x_1^{a'_1}y_1^{b'_1}\cdots x_g^{a'_g}y_g^{b'_g}\in \pi_1(S)^{\mathrm{Ab}}$.
Considering the coefficient of $ZZ'u\otimes 1$ for $u=x_1^{c_1}y_1^{d_1}\cdots x_g^{c_g}y_g^{d_g}$, we obtain
\begin{eqnarray}
&&(a_1b'_1-b_1a'_1+\cdots +a_gb'_g-b_ga'_g)C_{ZZ'u}^{ZZ'} \label{eq:g}\\
&=&((c_1+a_1)b'_1-(d_1+b_1)a'_1+\cdots +(c_g+a_g)b'_g-(d_g+b_g)a'_g)C_{Zu}^Z \nonumber \\
&&+(a_1(d_1+b'_1)-b_1(c_1+a'_1)+\cdots +a_g(d_g+b'_g)-b_g(c_g+a'_g))C_{Z'u}^{Z'}.\nonumber
\end{eqnarray}

If $u=1$, the map
$C:\mathbb{Z}^{2g}\rightarrow \mathbb{R},
(a_1, b_1, \cdots , a_g, b_g)\mapsto C_{x_1^{a_1}y_1^{b_1}\cdots x_g^{a_g}y_g^{b_g}}^{x_1^{a_1}y_1^{b_1}\cdots x_g^{a_g}y_g^{b_g}}$
is a homomorphism by Lemma \ref{lem:k0hom}. 
Hence
$(\Delta -\Delta_C)(Z)=\sum_{u\neq Z\in\pi_1(S)^{\mathrm{Ab}}} C_{u}^Z u\wedge 1$
for all $Z\in\pi_1(S)^{\mathrm{Ab}}$.

So we can assume $u\neq 1$.
Set
$C_Z:=C_{Zu}^Z$.
We will show
\begin{eqnarray}
C_{x_1^{a_1}y_1^{b_1}\cdots x_g^{a_g}y_g^{b_g}} = a_1C_{x_1}+b_1C_{y_1}+\cdots +a_g C_{x_g}+b_g C_{y_g}
\label{statement1}
\end{eqnarray}
for all $(a_1, b_1, \cdots, a_g, b_g)\in\mathbb{Z}^{2g}$
and $d_iC_{x_j}=d_jC_{x_i}$, $c_iC_{x_j}=-d_jC_{y_i}$, and $c_iC_{y_j}=c_jC_{y_i}$
for all $i,j\in\{1, \cdots , g\}$, by the induction on $g>0$.
In other words, $C_Z Zu\wedge 1=\frac{C_{x_1}}{d_1}d(u\wedge 1)(Z)$ for all $Z\in\pi_1(S)^{\mathrm{Ab}}$ when $d_1\neq 0$.
In the case $g=1$,
\begin{eqnarray}
(ab'-ba')C_{x^{a+a'}y^{b+b'}}=((c+a)b'-(d+b)a')C_{x^ay^b}+(a(d+b')-b(c+a'))C_{x^{a'}y^{b'}}.
\label{eq:g1}
\end{eqnarray}
It suffices to show (\ref{statement1}) in the case $c=c_1\neq 0$ and $d=d_1\neq 0$.
In fact, suppose the statement holds true for $c\neq 0, d\neq 0$.
For $A\in\mathrm{SL}(2, \mathbb{Z})$,
$(ab'-ba')C_{AZZ'}=((a+c')b'-(b+d')a')C_{AZ}+(a(b'+d')-b(a'+c'))C_{AZ'}$
where $A \left( \begin{array}{c} c' \\ d' \end{array} \right)=\left( \begin{array}{c} c \\ d \end{array} \right)$.
Set $C'_{Z}:=C_{AZ}$ and if $c'\neq 0$ and $d'\neq 0$, we obtain
$C'_{x^ay^b}=aC'_x+bC'_y$ and
$c'C'_x=-d'C'_y$.
If $c\neq 0, d=0$, by substituting $A=\left(\begin{array}{cc} 1 & 0  \\ -1 & 1 \end{array}\right)$, we have $\left( \begin{array}{c} c' \\ d' \end{array}\right)=\left(\begin{array}{c} c\\ c\end{array}\right)$.
By the assumption, we obtain
$C'_{x^ay^b}=aC'_x+bC'_y$ and
$cC'_x=-cC'_y$.
Therefore
\begin{eqnarray*}
C_{x^ay^b}
&=&C'_{x^ay^{a+b}}=aC'_x+(a+b)C'_y=aC'_{xy}+bC'_y\\
&=&aC_x+bC_y,
\end{eqnarray*}
and since $C_x=C'_{xy}=C'_x+C'_y=0, d=0$, we get
$cC_x=-dC_y$.
Similar argument is valid for the case $c=0$ and $d\neq 0$.

Suppose $c\neq 0$ and $d\neq 0$.
Substituting $b=b'=0$ into (\ref{eq:g1}), we have
$0=-da'C_{x^a}+daC_{x^{a'}}$
for all $a, a'\in\mathbb{Z}$.
Since $d\neq 0$, we obtain $C_{x^a}=aC_x$
Similarly, we have
$C_{y^b}=bC_y$.
Substituting $b=0, a'=0$ into (\ref{eq:g1}), we have
\begin{eqnarray*}
ab'C_{x^ay^{b'}}&=&(c+a)b'C_{x^a}+a(d+b')C_{y^{b'}}\\
&=&ab'((c+a)C_x+(d+b')C_y).
\end{eqnarray*}
So for $a, b\in\mathbb{Z}$ satisfying $ab\neq 0$,
\begin{eqnarray*}
C_{x^ay^b}&=&(c+a)C_x+(d+b)C_y\\
&=& \frac{C_x}{d}(ad-bc)-\frac{C_x}{d}(ad-bc)+(c+a)C_x+(d+b)C_y\\
&=& \frac{C_x}{d}(ad-bc)+(d+b)(\frac{c}{d}C_x+C_y).
\end{eqnarray*}
For $a, b, a', b'\in\mathbb{Z}$ satisfying $ab\neq 0, a'b'\neq 0$ and $(a+a')(b+b')\neq 0$,
applying the above equation to (\ref{eq:g1}), we have
\begin{eqnarray*}
&&(ab'-ba')(d+b+b')(\frac{c}{d}C_x+C_y)\\
&=&((c+a)b'-(d+b)a')(d+b)(\frac{c}{d}C_x+C_y)+(a(d+b')-b(c+a'))(d+b')(\frac{c}{d}C_x+C_y).
\end{eqnarray*}
Hence 
$d(2(ab'-ba')+(a-a')d-(b-b')c)(\frac{c}{d}C_x+C_y)=0$.
Now we choose $(a_0, b_0, a'_0, b'_0)$ satisfying $2(a_0b'_0-b_0a'_0)+(a_0-a'_0)d-(b_0-b'_0)c\neq 0$, for example $a_0=1, b_0=d, a'_0=2, b'_0=d$ as in Figure \ref{fig:example_xy_x^2y}. Then we obtain
$\frac{c}{d}C_x+C_y=0$.
Hence
$C_{x^ay^b}=\frac{C_x}{d}(ad-bc)=aC_x+bC_y$.
This proves (\ref{statement1}) for $g=1$.

\begin{figure}[htbp]
  \begin{center}
    \includegraphics[clip, width=6cm]{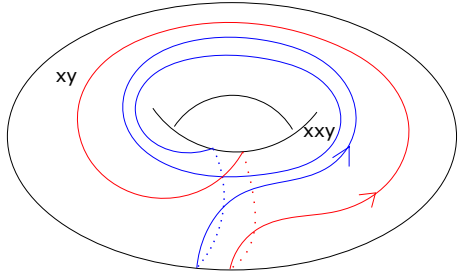}
    \caption{$xy^d$ and $x^2y^d$ when $d=1$.}
    \label{fig:example_xy_x^2y}
  \end{center}
\end{figure}

Consider the case $g\geq 2 $.
For $i\neq j$, substituting $(a_1, b_1, \cdots , a_g, b_g)=(0, \cdots, 0, \stackrel{2i-1}{\breve{1}}, 0, \cdots, 0)$ and $(a'_1, b'_1, \cdots, a'_g, b'_g)=(0, \cdots, 0, \stackrel{2j-1}{\breve{1}}, 0, \cdots, 0)$ into (\ref{eq:g}), we obtain
$0=-d_jC_{x_i}+d_iC_{x_j}$.
Similarly we have $0=c_jC_{x_i}+d_iC_{x_j}$ and $0=c_jC_{x_i}+d_iC_{y_j}$.
When $i=j$, from the calculation for $g=1$, we have $0=c_iC_{x_j}+d_jC_{y_i}$.

Now we assume (\ref{statement1}) for $g-1$.
Recall $(c_1, d_1, \cdots, c_g, d_g)\neq (0, \cdots, 0)$. Hence we assume $d_1\neq 0$. Similar argument holds for other cases.
By the inductive assumption, we have
$C_{x_1^{a_1}y_1^{b_1}\cdots x_{g-1}^{a_{g-1}}y_{g-1}^{b_{g-1}}} = \frac{C_{x_1}}{d_1}(a_1d_1-b_1c_1+\cdots +a_{g-1}d_{g-1}-b_{g-1}c_{g-1})$ and
$C_{x_g^{a_g}y_g^{b_g}}=a_gC_{x_g}+b_gC_{y_g}
= \frac{C_{x_1}}{d_1}(a_gd_g-b_gc_g)$.

If $(a_1, b_1, \cdots , a_{g-1}, b_{g-1})=(0, \cdots , 0)$ or $(a_g, b_g)=(0, 0)$, then (\ref{statement1}) holds true.
When $a_{g-1}\neq 0$ and $a_g\neq 0$, we have
\begin{eqnarray*}
&&a_{g-1}C_{x_1^{a_1}y_1^{b_1}\cdots x_g^{a_g}y_g^{b_g}}=(c_{g-1}+a_{g-1}+c_gb_g-d_ga_g)C_{x_1^{a_1}y_1^{b_1}\cdots x_{g-1}^{a_{g-1}}y_{g-1}^{b_{g-1}-1}}\\
&&\ +(a_1d_1-b_1c_1+\cdots+a_{g-2}d_{g-2}-b_{g-2}c_{g-2} \\
&&\ \ +a_{g-1}(d_{g-1}+1)-(b_{g-1}-1)c_{g-1})C_{y_{g-1}x_g^{a_g}y_g^{b_g}},\\
&&a_gC_{y_{g-1}x_g^{a_g}y_g^{b_g}} = (c_{g-1}+c_g+a_g)C_{x_g^{a_g}y_g^{b_g-1}}+(a_g(d_g+1)-(b_g-1)c_g)C_{y_{g-1}y_g}.
\end{eqnarray*}
From straight-forward calculation, we have
\begin{eqnarray}
&&a_{g-1}a_gC_{x_1^{a_1}y_1^{b_1}\cdots x_g^{a_g}y_g^{b_g}}\nonumber\\
&=&a_{g-1}a_g \frac{C_{x_1}}{d_1}(a_1d_1-b_1c_1+\cdots +a_gd_g-b_gc_g)\nonumber\\
&&+(a_1d_1-b_1c_1+\cdots +a_{g-2}d_{g-2}-b_{g-2}c_{g-2}+a_{g-1}(d_{g-1}+1)-(b_{g-1}-1)c_{g-1})\nonumber\\
&&\cdot(c_gb_g-d_ga_g-c_g-a_g)(C_{y_{g-1}y_g}-\frac{C_{x_1}}{d_1}(-C_{g-1}-C_g)).\label{eq:ag-1ag}
\end{eqnarray}
If $c_g=0$ and $d_g=-1$, we obtain
$C_{x_1^{a_1}y_1^{b_1}\cdots x_g^{a_g}y_g^{b_g}}=\frac{C_{x_1}}{d_1}(a_1d_1-b_1c_1+\cdots +a_gd_g-b_gc_g)$.
So we can assume $c_g\neq 0$ or $d_g\neq -1$.
By applying (\ref{eq:ag-1ag}) to (\ref{eq:g}),
\begin{eqnarray}
&&(\sum_{i=1}^g a_ib'_i-b_ia'_i)\cdot\frac{C_{y_{g-1}y_g}-\frac{C_{x_1}}{d_1}(-C_{g-1}-C_g)}{(a_{g-1}+a'_{g-1})(a_g+a'_g)}((a_1+a'_1)d_1-(b_1+b'_1)c_1+\cdots \nonumber\\
&&+(a_{g-1}+a'_{g-1})d_{g-1}-(b_{g-1}+b'_{g-1}-1)c_{g-1}+a_{g-1}+a'_{g-1})\nonumber\\&&\cdot (c_g(b_g+b'_g)-d_g(a_g+a'_g)-c_g-(a_g+a'_g))\nonumber\\
&=&(\sum_{i=1}^g (c_i+a_i)b'_i-(d_i+b_i)a'_i)\cdot\frac{C_{y_{g-1}y_g}-\frac{C_{x_1}}{d_1}(-C_{g-1}-C_g)}{a_{g-1}a_g}(a_1d_1-b_1c_1+\cdots\nonumber\\
&&+a_{g-1}d_{g-1}-(b_{g-1}-1)c_{g-1}+a_{g-1})(c_gb_g-d_ga_g-c_g-a_g)\nonumber\\
&&+(\sum_{i=1}^g a_i(d_i+b'_i)-b_i(c_i+a'_i))\frac{C_{y_{g-1}y_g}-\frac{C_{x_1}}{d_1}(-C_{g-1}-C_g)}{a'_{g-1}a'_g}(a'_1d_1-b'_1c_1+\cdots\nonumber\\
&&+a'_{g-1}d_{g-1}-(b'_{g-1}-1)c_{g-1}+a'_{g-1})(c_gb'_g-d_ga'_g-c_g-a'_g).
\label{eq:long_equation}
\end{eqnarray}
The coefficient of $a_1^2$ in (\ref{eq:long_equation}) as a polynomial of $a_1$ is given by
\begin{eqnarray*}
&&b_1'\cdot\frac{C_{y_{g-1}y_g}-\frac{C_{x_1}}{d_1}(-C_{g-1}-C_g)}{(a_{g-1}+a'_{g-1})(a_g+a'_g)}d_1(c_g(b_g+b'_g)-d_g(a_g+a'_g)-c_g-(a_g+a'_g))\\
&=&b_1'\cdot\frac{C_{y_{g-1}y_g}-\frac{C_{x_1}}{d_1}(-C_{g-1}-C_g)}{a_{g-1}a_g}d_1(c_gb_g-d_ga_g-c_g-a_g).
\end{eqnarray*}
This equation holds for all $(a_1, b_1, \cdots, a_g, b_g)$ and $(a'_1, b'_1, \cdots, a'_g, b'_g)\in\mathbb{Z}^{2g}$ satisfying $a_{g-1}\neq 0, a_g\neq 0, a'_{g-1}\neq 0, a'_g\neq 0, a_{g-1}+a'_{g-1}\neq 0$ and $a_g+a'_g\neq 0$.
Since $d_1\neq 0$ and $(c_g, d_g)\neq (0, -1)$, we obtain
$C_{y_{g-1}y_g}-\frac{C_{x_1}}{d_1}(-C_{g-1}-C_g)=0$.
Therefore
$C_{x_1^{a_1}y_1^{b_1}\cdots x_g^{a_g}y_g^{b_g}}=\frac{C_{x_1}}{d_1}(a_1d_1-b_1c_1+\cdots +a_gd_g-b_gc_g)$.
This proves the map $\mathrm{Hom}(\mathbb{Z}^{2g}, \mathbb{R})\rightarrow H^1(W_1(g), W_1(g)\wedge 1)$ is injective.

We show the map $H^1(W_1(g), W_1(g)\wedge 1)\rightarrow H^1(W_1(g), W_1(g)\wedge W_1(g))$ is injective.
Let $\Delta\in Z^1(W_1(g), W_1(g)\wedge 1)$ be a coboundary, that is,
$\Delta=\sum_{u, v\in\pi_1(S)^{\mathrm{Ab}}}C_{u, v}d(u\wedge v)$
for some $C_{u, v}\in\mathbb{R}$.
For any $Z\in\pi_1(S)^{\mathrm{Ab}}$, 
$\sum_{u, v\in\pi_1(S)^{\mathrm{Ab}}}C_{u, v}(\{Z, u\}\wedge v+u\wedge\{Z, v\})=\Delta(Z)\in W_1(g)\wedge 1$. 
If $\{Z, u\}\wedge v\neq 0$ and $Zu \wedge v\in W_1(g)\wedge 1$, then we obtain $v=1$.
Therefore we can take $w\in W_1(g)$ such that
$\Delta=d(w\wedge 1)$.
This proves Theorem \ref{thm:wedge1}.
\end{proof}

\end{subsection}
\begin{subsection}{Computation of $H^1(W_1(g), W_1(g)\otimes W_1(g))$}\label{section2}
In this section we show the fact $\Delta\in Z^1(W_1(g), W_1(g)\otimes W_1(g))$ is determined by the values $\Delta(1), \Delta(x_1), \Delta(y_1), \cdots, \Delta(x_g)$ and $\Delta(y_g)$.
By the decomposition of $W_1(g)$-modules  $W_1(g)\otimes W_1(g)=\mathbb{R}(1\otimes 1)\oplus(W_1(g)'\otimes 1)\oplus(1\otimes W_1(g)')\oplus(W_1(g)'\otimes W_1(g)')$, we have a decomposition of $\mathbb{R}$-modules
\begin{align*} H^1(W_1(g), W_1(g)\otimes W_1(g))&\cong H^1(W_1(g), \mathbb{R})\oplus H^1(W_1(g), W_1(g)')\\ \ &\oplus H^1(W_1(g), W_1(g)')\oplus H^1(W_1(g), W_1(g)'\otimes W_1(g)').\end{align*}
We will show $H^1(W_1(g), \mathbb{R})\cong\mathbb{R}$, $H^1(W_1(g), W_1(g)')\cong\mathrm{Hom}(\mathbb{Z}^{2g}, \mathbb{R})$ and
$H^1(W_1(g), W_1(g)'\otimes W_1(g)')\cong 0$.
This implies there is an isomorphism $H^1(W_1(g), W_1(g)\otimes W_1(g))\cong \mathbb{R}\times\mathrm{Hom}(\mathbb{Z}^{2g}, \mathbb{R})^2$.
As a corollary, we can show the map $\mathrm{Hom}(\mathbb{Z}^{2g}, \mathbb{R})\rightarrow H^1(W_1(g), W_1(g)\wedge W_1(g))$ is surjective.

Let $\Delta\in Z^1(W_1(g), W_1(g)\otimes W_1(g))$ be given by
\begin{align}
\Delta(Z)=\sum_{u, v\in\pi_1(S)^{\mathrm{Ab}}}C^Z_{u, v}u\otimes v\label{eq:coefz}
\end{align}
for $Z\in\pi_1(S)^{\mathrm{Ab}}$.
Since $\Delta$ is compatible with the bracket, we have
$\Delta\{Z_1, Z_2\}=\{\Delta(Z_1), Z_2\}+\{Z_1, \Delta(Z_2)\}$.
Considering the coefficients of $u\otimes v$, we obtain
\begin{align}
i(Z_1, Z_2)C^{Z_1Z_2}_{u, v}&=i(u, Z_2)C^{Z_1}_{Z_2^{-1}u, v}+i(v, Z_2)C^{Z_1}_{u, Z_2^{-1}v}\nonumber\\
&+i(Z_1, u)C^{Z_2}_{Z_1^{-1}u, v}+i(Z_1, v)C^{Z_2}_{u, Z_1^{-1}v}
\label{compatible_condition}
\end{align}
for all $Z_1, Z_2\in\pi_1(S)^{\mathrm{Ab}}$ and $u, v\in\pi_1(S)^{\mathrm{Ab}}$.

\begin{prop}\label{thm0}
There is an injective homomorphism
\begin{align*}
Z^1(W_1(g), W_1(g)\otimes W_1(g))\rightarrow (W_1(g)\otimes W_1(g))^{1+2g},\\
\Delta \mapsto (\Delta(1), \Delta(x_1), \Delta(y_1), \cdots, \Delta(x_g), \Delta(y_g)).
\end{align*}
\end{prop}

\begin{proof}

Let $\Delta\in Z^1(W_1(g), W_1(g)\otimes W_1(g))$ be the 1-cocycle given by (\ref{eq:coefz}).
Suppose $\Delta(1)=0$ and $\Delta(x_i)=\Delta(y_i)=0$ for all $i=1, \cdots, g$.
Substituting $Z_1=x_i, Z_2=y_i$ into (\ref{compatible_condition}), we have $C^{x_iy_i}_{u, v}=0$ for all $u, v\in\pi_1(S)^{\mathrm{Ab}}$. Hence $\Delta(x_iy_i)=0$.
Since $\Delta$ is compatible with the bracket,
we have
\begin{align*}
\Delta(x_i^{a_i}y_i)&=\{x_i, \Delta(x_i^{a_i-1}y_i)\}=\cdots=\{x_i, \cdots \{x_i, \Delta(x_iy_i)\}\cdots\}=0,\\
\Delta(x_i^{a_i}y_i^{b_i})&=a_i^{-1}\{\Delta(x_i^{a_i}y_i^{b_i-1}), y_i\}=\cdots \\& =a_i^{-(b_i-1)}\{\cdots\{\{\Delta(x_i^{a_i}y_i), y_i\}, y_i\} \cdots, y_i\}=0
\end{align*}
for all $a_i, b_i>0$ and $i=1, \cdots, g$.

Next we show $\Delta(x_i^{-1})=0$.
Substituting $Z_1=x_i y_i, Z_2=x_i^{-1}$ into (\ref{compatible_condition}), we have
$0=i(x_iy_i, u)C^{x_i^{-1}}_{(x_iy_i)^{-1}u, v}+i(x_iy_i,v)C^{x_i^{-1}}_{u, (x_iy_i)^{-1}v}$.
Consider the following equations
\begin{eqnarray*}
0=i(x_iy_i, u)C^{x_i^{-1}}_{(x_iy_i)^{n-1}u, (x_iy_i)^{-n}v}+i(x_iy_i,v)C^{x_i^{-1}}_{(x_iy_i)^{n}u, (x_iy_i)^{-n-1}v}
\end{eqnarray*}
for all $n\in\mathbb{Z}$.
If 
\begin{itemize}
\item$i(x_iy_i, u)\neq 0\ \text{and}\ i(x_iy_i, v)=0,\ \text{or}$ 
\item$i(x_iy_i, u)=0\ \text{and}\ i(x_iy_i, v)\neq 0$,
\end{itemize}
then
$C^{x_i^{-1}}_{(x_iy_i)^{n-1}u, (x_iy_i)^{-n}v}=0$
for all $n\in\mathbb{Z}$.
Hence we get
$C^{x_i^{-1}}_{u', v'}=0$
if
\begin{itemize}
\item$i(x_iy_i, u')\neq 0\ \text{and}\ i(x_iy_i, v')=0,\ \text{or}$ 
\item$i(x_iy_i, u')=0\ \text{and}\ i(x_iy_i, v')\neq 0$,
\end{itemize}
by substituting $u=x_iy_iu', v=v', n=0$ into the above equation.

If $i(x_iy_i, u)\neq 0$ and $i(x_iy_i, v)\neq 0$, then there exists $N>0$ such that $C^{x_i^{-1}}_{(x_iy_i)^{n-1}u, (x_iy_i)^{-n}v}=0$ for all $n>N$.
Hence
\begin{eqnarray*}
C^{x_i^{-1}}_{(x_iy_i)^{N-1}u, (x_iy_i)^{-N}v}&=&-\frac{i(x_iy_i, v)}{i(x_iy_i, u)}C^{x_i^{-1}}_{(x_iy_i)^Nu, (x_iy_i)^{-(N+1)}v} =0,\\
C^{x_i^{-1}}_{(x_iy_i)^{N-2}u, (x_iy_i)^{-(N-1)}v}&=&-\frac{i(x_iy_i, v)}{i(x_iy_i, u)}C^{x_i^{-1}}_{(x_iy_i)^{N-1}u, (x_iy_i)^{-N}v} =0
\end{eqnarray*}
and we obtain
$C^{x_i^{-1}}_{(x_iy_i)^{n-1}u, (x_iy_i)^{-n}v}=0$ for all $n\in\mathbb{Z}$ by induction.  Hence $C^{x_i^{-1}}_{u, v}=0$ when $i(x_iy_i, u)\neq 0$ or $i(x_iy_i, v)\neq 0$.

Substituting $(Z_1, Z_2)=(x_i^{-1}, x_j), (x_i^{-1}, y_j)$ into (\ref{compatible_condition}), where $x_j\in\{x_1, \cdots, x_g\}$ and $y_j\in\{y_1, \cdots, \hat y_i, \cdots, y_g\}$, we have
\begin{align*}
0=i(u, x_j)C^{x_i^{-1}}_{x_j^{-1}u, v}+i(v, x_j)C^{x_i^{-1}}_{u, x_j^{-1}v},\\
0=i(u, y_j)C^{x_i^{-1}}_{y_j^{-1}u, v}+i(v, y_j)C^{x_i^{-1}}_{u, y_j^{-1}v}.
\end{align*}
By a similar argument, we have $C^{x_i^{-1}}_{u, v}=0$ if there exists
\begin{itemize}\item $j\in\{1, \cdots, g\}$ such that $i(u, x_j)\neq 0$ or $i(v, x_j)\neq 0$ or
\item$j\in\{1, \cdots\hat i, \cdots, g\}$ such that $i(u, y_j)\neq 0$ or $i(v, y_j)\neq 0$.
\end{itemize}
Substituting $Z_1=y_i^{-1}, Z_2=x_iy_i$ and $u=v=1$ into (\ref{compatible_condition}), we have
\begin{eqnarray*}
C^{x_i}_{1, 1}&=&i(1, x_iy_i)C^{y_i^{-1}}_{(x_iy_i)^{-1}, 1}+i(1, x_iy_i)C^{y_i^{-1}}_{1, (x_iy_i)^{-1}}+i(y_i^{-1}, 1)C^{x_iy_i}_{y_i, 1}\\&&\ +i(y_i^{-1}, 1)C^{x_iy_i}_{1, y_i}=0.
\end{eqnarray*}
Hence $\Delta(x_i^{-1})=0$.
Similarly, we have $\Delta(y_i^{-1})=0$.
Since $\Delta(1)=0$, we obtain
$\Delta(x_i^{a_i}y_i^{b_i})=0$
for all $a_i, b_i\in\mathbb{Z}$.

Now it suffices to show
$\Delta(z_iw_j)=0$
for all $i\neq j, z_i\in\{x_i, y_i\}$ and $w_j\in\{x_j, y_j\}$.
Suppose $\Delta(z_iw_j)=0$ for all $i\neq j, z_i\in\{x_i, y_i\}$ and $w_j\in\{x_j, y_j\}$.
We can show $\Delta$=0 by induction on $g$.
If $g=1$, we have $\Delta=0$ from the above argument.
Suppose $\Delta(x_1^{a_1}y_1^{b_1}\cdots x_{k}^{a_{k}}y_{k}^{b_{k}})=0$ for all $(a_1, b_1, \cdots a_{k}, b_{k})\in\mathbb{Z}^{2k}$ for some $1\leq k\leq g$.
If $(a_1, b_1, \cdots, a_{k}, b_{k})=(0, \cdots, 0)$ or $(a_{k+1}, b_{k+1})=(0, 0)$, we have
$\Delta(x_1^{a_1}y_1^{b_1}\cdots x_{k+1}^{a_{k+1}}y_{k+1}^{b_{k+1}})=0$.
So we may assume $(a_1, b_1, \cdots, a_{k}, b_{k})\neq(0, \cdots, 0)$ and $(a_{k+1}, b_{k+1})\neq(0, 0)$
In that case there exist $i\in\{1, \cdots , k\}, c_i\in\{a_i, b_i\}$ and $d_{k+1}\in\{a_{k+1}, b_{k+1}\}$ such that $c_i\neq 0$ and $d_{k+1}\neq0$. Suppose $c_i=a_{k}$ and $d_{k+1}=a_{k+1}$. Similar argument is valid for other cases.
\begin{eqnarray*}
&&\Delta(x_1^{a_1}y_1^{b_1}\cdots x_g^{a_{k+1}}y_g^{b_{k+1}})\\
&=&a_k^{-1}(\{\Delta(x_1^{a_1}y_1^{b_1}\cdots x_{k}^{a_{k}}y_{k}^{b_{k}-1}), y_{k}x_{k+1}^{a_{k+1}}y_{k+1}^{b_{k+1}}\}
+\{x_1^{a_1}y_1^{b_1}\cdots x_{k}^{a_{k}}y_{k}^{b_{k}-1}, \Delta(y_{k}x_{k+1}^{a_{k+1}}y_{k+1}^{b_{k+1}})\})\\
&=&a_k^{-1}\{x_1^{a_1}y_1^{b_1}\cdots x_{k}^{a_{k}}y_{k}^{b_{k}-1}, \Delta(y_{k}x_{k+1}^{a_{k+1}}y_{k+1}^{b_{k+1}})\}\\
&=&a_k^{-1}a_{k+1}^{-1}(\{x_1^{a_1}y_1^{b_1}\cdots x_{k}^{a_{k}}y_{k}^{b_{k}-1},
\{\Delta(x_{k+1}^{a_{k+1}}y_{k+1}^{b_{k+1}-1}), y_{k}y_{k+1}\}+\{x_{k+1}^{a_{k+1}}y_{k+1}^{b_{k+1}-1}, \Delta(y_{k}y_{k+1})\}\})\\
&=&0.
\end{eqnarray*}
We will show $\Delta(z_iw_j)=0$ for all $i\neq j, z_i\in\{x_i, y_i\}$ and $w_j\in\{x_j, y_j\}$.
Substituting $Z_1=z_i^{a_i}w_j^{b_j}$ and $Z_2=\zeta$ into (\ref{compatible_condition}), where $\zeta\in\{x_1, y_1, \cdots, x_g, y_g\}$ and $i(z_i^{a_i}w_j^{b_j}, \zeta)=0$, we have
$0=i(u, \zeta)C^{z_i^{a_i}w_j^{b_j}}_{\zeta^{-1}u, v}+i(v,\zeta)C^{z_i^{a_i}w_j^{b_j}}_{u, \zeta^{-1}v}$.
So if $i(u, \zeta)\neq 0$ or $i(v, \zeta)\neq 0$, then we have $C^{z_i^{a_i}w_j^{b_j}}_{u, v}=0$.
Therefore if $C^{z_i^{a_i}w_j^{b_j}}_{u, v}\neq 0$ for some $a_i$ and $b_j$, then $u, v\in \langle z_i, w_j\rangle\subset \pi_1(S)^{\mathrm{Ab}}$.
Let us define $\overline z_i$ by
\[\overline z_i=\left\{ \begin{array}{ll} y_i & \text{if}\ z_i=x_i \\ x_i & \text{if}\ z_i=y_i\end{array}\right. .\] 
Substituting $(Z_1, Z_2)=(z_iw_j, \overline z_i)$ and $(z_i, w_j\overline z_i)$ into (\ref{compatible_condition}),
we have
\begin{eqnarray}
i(u, \overline z_i)C^{z_iw_j}_{\overline z_i^{-1}u, v}+i(v, \overline z_i)C^{z_iw_j}_{u, \overline z_i^{-1}v} =i(z_i, \overline z_i)C^{z_i\overline z_iw_j}_{u, v}=i(z_i, u)C^{w_j\overline z_i}_{z_i^{-1}u, v}+i(z_i, v)C^{w_j\overline z_i}_{u, z_i^{-1}v}.\label{eq:zw0}
\end{eqnarray}
Substituting $(\overline z_i^{-1}u, v)=(z_i^{c_i}w_j^{d_j}, z_i^{c'_i}w_j^{d'_j})$ and $(u, \overline z_i^{-1}v)=(z_i^{c_i}w_j^{d_j}, z_i^{c'_i}w_j^{d'_j})$ into (\ref{eq:zw0}),
we have 
\begin{align*}
c_i i(z_i, \overline z_i)C^{z_iw_j}_{z_i^{c_i}w_j^{d_j}, z_i^{c'_i}w_j^{d'_j}}=i(z_i, \overline z_i)C^{w_j\overline z_i}_{\overline z_i z_i^{c_i-1}w_j^{d_j}, z_i^{c'_i}w_j^{d'_j}},\\
c'_i i(z_i, \overline z_i)C^{z_iw_j}_{z_i^{c_i}w_j^{d_j}, z_i^{c'_i}w_j^{d'_j}}=i(z_i, \overline z_i)C^{w_j\overline z_i}_{z_i^{c_i}w_j^{d_j}, \overline z_iz_i^{c'_i-1}w_j^{d'_j}}.
\end{align*}
If $(c_i, c'_i)\neq (1, 0), (0, 1), (0, 0)$, then
\begin{itemize}
\item $c_i\neq 0$ and $C^{w_j\overline z_i}_{\overline z_i z_i^{c_i-1}w_j^{d_j}, z_i^{c'_i}w_j^{d'_j}}=0$, or 
\item $c'_i\neq 0$ and $C^{w_j\overline z_i}_{z_i^{c_i}w_j^{d_j}, \overline z_iz_i^{c'_i-1}w_j^{d'_j}}=0$. 
\end{itemize}
Therefore $C^{z_iw_j}_{z_i^{c_i}w_j^{d_j}, z_i^{c'_i}w_j^{d'_j}}=0$ for $(c_i, c'_i)\neq (1, 0), (0, 1), (0, 0)$. 

Substituting $Z_1=z_iw_j$ and $Z_2=\overline z_i^{\epsilon_i}\overline w_j^{\epsilon_j}$ into (\ref{compatible_condition}), where $\epsilon_i, \epsilon_j\in\{1, -1\}$, $i(z_i, \overline z_i^{\epsilon_i})=1$ and $i(w_j, \overline w_j^{\epsilon_j})=-1$, we have
\begin{eqnarray}
0&=&i(u, \overline z_i^{\epsilon_i}\overline w_j^{\epsilon_j})C^{z_iw_j}_{(\overline z_i^{\epsilon_i}\overline w_j^{\epsilon_j})^{-1}u, v}+i(v, \overline z_i^{\epsilon_i}\overline w_j^{\epsilon_j})C^{z_iw_j}_{u, (\overline z_i^{\epsilon_i}\overline w_j^{\epsilon_j})^{-1}v}\nonumber\\&&+i(z_iw_j, u)C^{\overline z_i^{\epsilon_i}\overline w_j^{\epsilon_j}}_{(z_iw_j)^{-1}u, v}+i(z_iw_j, v)C^{\overline z_i^{\epsilon_i}\overline w_j^{\epsilon_j}}_{u, (z_iw_j)^{-1}v}.\label{eq:zw}
\end{eqnarray}
Substituting $((\overline z_i^{\epsilon_i}\overline w_j^{\epsilon_j})^{-1}u, v)=(z_i^{c_i}w_j^{d_j}, z_i^{c'_i}w_j^{d'_j})$ and $(u, (\overline z_i^{\epsilon_i}\overline w_j^{\epsilon_j})^{-1}v)=(z_i^{c_i}w_j^{d_j}, z_i^{c'_i}w_j^{d'_j})$ into (\ref{eq:zw}), we obtain
\begin{align*}
0=(c_i-d_j)C^{z_iw_j}_{z_i^{c_i}w_j^{d_j}, z_i^{c'_i}w_j^{d'_j}},\\
0=(c'_i-d'_j)C^{z_iw_j}_{z_i^{c_i}w_j^{d_j}, z_i^{c'_i}w_j^{d'_j}}.
\end{align*}
So if $c_i-d_j\neq 0$ or $c'_i-d'_j\neq 0$, then $C^{z_iw_j}_{z_i^{c_i}w_j^{d_j}, z_i^{c'_i}w_j^{d'_j}}=0$.

Therefore, if there exists $(u, v)$ such that $C^{z_iw_j}_{u, v}\neq 0$, then $(u, v)=(z_iw_j, 1), (1, z_iw_j), (1, 1)$.
Substituting $Z_1=z_iw_j\overline w_j, Z_2=\overline w_j^{-1}, u=1$ and $v=1$ into (\ref{compatible_condition}),
\[C^{z_iw_j}_{1, 1}=\frac{1}{i(w_j, \overline w_j^{-1})}(i(1, \overline w_j^{-1})C^{z_iw_j\overline w_j}_{\overline w_j, 1}+i(1, \overline w_j^{-1})C^{z_iw_j\overline w_j}_{1, \overline w_j })=0.\]
By a similar argument to the proof of Theorem \ref{thm:hom}, we obtain
$C^{z_iw_j}_{z_iw_j, 1}=C^{z_i}_{z_i, 1}+C^{w_j}_{w_j, 1}=0$ and
$C^{z_iw_j}_{1, z_iw_j}=C^{z_i}_{1, z_i}+C^{w_j}_{1, w_j}=0$.
This implies $\Delta(z_iw_j)=0$.
This proves Proposition \ref{thm0}.
\end{proof}

\begin{lemm}\label{lem:injective}
There is an isomorphism $\mathbb{R}\rightarrow H^1(W_1(g), \mathbb{R}(1\otimes 1)), r\mapsto [\delta_0(r)]$, where \[\delta_0(r)(Z)=\begin{cases}r\cdot 1\otimes 1\ \text{if}\ Z=1\\ 0\ \text{otherwise}\end{cases}\] for $Z\in\pi_1(S)^{\mathrm{Ab}}$.
\end{lemm}
\begin{proof}
$\delta_0(r)\in Z^1(W_1(g), \mathbb{R}(1\otimes 1))$ is clear.
Since $(d\alpha)(1)=0$ for all $\alpha\in W_1(g)$, the map is injective.

Let $\Delta\in Z^1(W_1(g), \mathbb{R}(1\otimes 1))$ be given by $\Delta(Z)=C^Z_{1, 1}1\otimes 1$ for $Z\in\pi_1(S)^{\mathrm{Ab}}$.
For all $Z\neq 1\in\pi_1(S)^{\mathrm{Ab}}$, there exists $Z'\in\pi_1(S)^{\mathrm{Ab}}$ such that $i(Z, Z')\neq 0$.
Substituting $Z_1=Z'^{-1}$, $Z_2=ZZ'$ and $u=v=1$ into (\ref{compatible_condition}), we have
\[ i(Z, Z')C^Z_{1, 1}=i(1, ZZ')C^{Z'^{-1}}_{(ZZ')^{-1}, 1}+i(1, ZZ')C^{Z'^{-1}}_{1, (ZZ')^{-1}}+i(Z'^{-1}, 1)C^{ZZ'}_{Z', 1}+i(Z'^{-1}, 1)C^{ZZ'}_{1, Z'}=0.\]
Hence $\Delta=\delta_0(C^1_{1, 1})$.
\end{proof}

By a similar argument to the proof of Theorem \ref{thm:hom} and Theorem \ref{thm:wedge1}, we obtain following Lemma.
\begin{lemm}\label{lem:otimes1}
For a map $k:\mathbb{Z}^{2g}\cong\pi_1(S)^{\mathrm{Ab}}\rightarrow \mathbb{R}$, we can define two maps
\begin{align*}
\Delta_k^l&: W_1(g) \rightarrow W_1(g) \otimes 1,
\Delta_k^l(Z)=k(Z)Z\otimes 1,\\
\Delta_k^r&: W_1(g) \rightarrow 1 \otimes W_1(g),
\Delta_k^r(Z)=k(Z)1\otimes Z,
\end{align*}
for $Z\in\pi_1(S)^{\mathrm{Ab}}$.
Then
\begin{enumerate}\item $\Delta_k^l, \Delta_k^r\in Z^1(W_1(g), W_1(g)\otimes W_1(g))$ if and only if there exists a homomorphism $k':\mathbb{Z}^{2g}\rightarrow \mathbb{R}$ such that $k|_{\mathbb{Z}^{2g}\backslash(0, \cdots, 0)}=k'|_{\mathbb{Z}^{2g}\backslash(0, \cdots, 0)}$.
\item We have two isomorphisms $\mathrm{Hom}(\mathbb{Z}^{2g}, \mathbb{R})\rightarrow H^1(W_1(g), W_1(g)'\otimes 1), k\mapsto [\Delta_k^l]$ and
$\mathrm{Hom}(\mathbb{Z}^{2g}, \mathbb{R})\rightarrow H^1(W_1(g), 1\otimes W_1(g)'), k\mapsto [\Delta_k^r]$.
\end{enumerate}
\end{lemm}
\begin{proof}
By a similar argument to the proof of Theorem \ref{thm:hom}, we have Lemma \ref{lem:otimes1}.1.
Let $\Delta\in Z^1(W_1(g), W_1(g)\otimes W_1(g))$ be given by $\Delta(Z)=\sum_{u, v\in\pi_1(S)^{\mathrm{Ab}}}C^Z_{u, v}u\otimes v$ for $Z\in\pi_1(S)^{\mathrm{Ab}}$.
Since
$\Delta\{Z, 1\}=\{\Delta(Z), 1\}+\{Z, \Delta(1)\}$, we have
$\{Z, \Delta(1)\}=0$ for all $Z\in\pi_1(S)^{\mathrm{Ab}}$.
For $Z=x_1^{a_1}y_1^{b_1}\cdots x_g^{a_g}y_g^{b_g}$, we obtain
\[0=\{Z, \Delta(1)\}
=\sum_{u, v\in\pi_1(S)^{\mathrm{Ab}}}C^1_{u, v}(i(Z, u)Zu\otimes v+i(Z, v)u\otimes Zv).\]
If we take $a_1, b_1, \cdots, a_g, b_g>\mathrm{max}_{C^1_{u, v}\neq 0}\{\mathrm{deg}_{x_i}u, \mathrm{deg}_{x_i}v, \mathrm{deg}_{y_i}u, \mathrm{deg}_{y_i}v\ |\ i=1, \cdots, g\}-\mathrm{min}_{C^1_{u, v}\neq 0}\{\mathrm{deg}_{x_i}u, \mathrm{deg}_{x_i}v, \mathrm{deg}_{y_i}u, \mathrm{deg}_{y_i}v\ |\ i=1, \cdots, g\}$,
then the coefficient of $Zu\otimes v$ is
$C^1_{u, v}i(Z, u)=0$
and the coefficient of $u\otimes Zv$ is
$C^1_{u, v}i(Z, v)=0$.
Hence $C^1_{u, v}=0$ for $(u, v)\neq (1, 1)$.
By a similar argument to the proof of Theorem \ref{thm:wedge1}, we have Lemma \ref{lem:otimes1}.2.
\end{proof}

\begin{theo}\label{thm:iso_g}
We have an isomorphism
\begin{align*}
\mathbb{R}\times\mathrm{Hom}(\mathbb{Z}^{2g}, \mathbb{R})^2\cong H^1(W_1(g), W_1(g)\otimes W_1(g)),
(r, k_l, k_r)\mapsto [\delta_0(r)+\Delta^l_{k_l}+\Delta^r_{k_r}].
\end{align*}
\end{theo} 

As a corollary, we obtain following isomorphism.
\begin{coro}
$H^1(W_1(g), W_1(g)\wedge W_1(g))\cong \mathrm{Hom}(\mathbb{Z}^{2g}, \mathbb{R})$.
\end{coro}
\begin{proof}
Let $\phi$ be the isomorphism in Theorem \ref{thm:iso_g}.
Define three linear maps as following.
\begin{align*}
\iota:\mathrm{Hom}(\mathbb{Z}^{2g}, \mathbb{R})\rightarrow \mathbb{R}\times\mathrm{Hom}(\mathbb{Z}^{2g}, \mathbb{R})^2, k\mapsto (0, k, -k),\\
s:W_1(g)\wedge W_1(g)\rightarrow W_1(g)\otimes W_1(g), u\wedge v\mapsto u\otimes v-v\otimes u,\\
p:W_1(g)\otimes W_1(g)\rightarrow W_1(g)\wedge W_1(g), u\otimes v\mapsto \frac{1}{2}u\wedge v.
\end{align*}
Consider the composition $p_*\circ \phi\circ\iota$, where
$p_*$ denotes the induced map $H^1(W_1(g), W_1(g)\otimes W_1(g))\rightarrow H^1(W_1(g), W_1(g)\wedge W_1(g))$.
Since $\{(r, k, k)\ |\ r\in\mathbb{R}, k\in\mathrm{Hom}(\mathbb{Z}^{2g}, \mathbb{R})\}\subset\ker (p_*\circ\phi)$, the composition $p_*\circ\phi\circ\iota$ is surjective.
Moreover $p_*\circ\phi\circ\iota$ is injective from $s_*\circ p_*\circ\phi\circ\iota=\phi\circ\iota$ and the injectivity of $s_*$.
\end{proof}

We begin by proving Theorem \ref{thm:iso_g} for the case $g=1$.
After that we will prove Theorem \ref{thm:iso_g} for a general genus $g$, by almost the same way for $g=1$.

We show $H^1(W_1(g), W_1(g)'\otimes W_1(g)')\cong 0$ by showing, for all $\Delta\in Z^1(W_1(g), W_1(g)'\otimes W_1(g)')$, there exists $\alpha\in W_1(g)'\otimes W_1(g)'$ such that $\Delta(x_i)=d\alpha(x_i)$ and $\Delta(y_i)=d\alpha(y_i)$ for all $1\leq i\leq g$. 
For $\Delta \in Z^1(W_1(g), W_1(g)'\otimes W_1(g)')$, we have $\Delta(1)=0$ from $\{\Delta(1), Z\}=0$ for all $Z\in \pi_1(S)^{\mathrm{Ab}}$.

\begin{lemm}\label{lem:y0}
Suppose $g=1$ and $\Delta\in Z^1(W_1(1), W_1(1)'\otimes W_1(1)')$.
If $\Delta(y)=0$, then there exists some element $\alpha\in W_1(1)'\otimes W_1(1)'$ such that the difference $\Delta':=\Delta-d\alpha$ satisfies $\Delta'(x)=0$ and $\Delta'(y)=0$.
\end{lemm}

\begin{proof}
Set
$\Delta(Z)=\sum_{u, v\in\pi_1(S)^{\mathrm{Ab}}}C^Z_{u, v}u\otimes v$
for $Z\in\pi_1(S)^{\mathrm{Ab}}$.
For $b>0$, we have
\begin{eqnarray}
&&\Delta(xy^b)=\{\Delta(xy^{b-1}), y\}=\{\{\Delta(xy^{b-2}), y\}, y\}=\cdots=\{\cdots\{\{\Delta(x), \overbrace{y\}, y\}, \cdots , y}^b\}\nonumber\\&=& \sum_{u, v\in\pi_1(S)^{\mathrm{Ab}}}C^x_{u, v}\sum_{k, l\geq 0, k+l=b}i(u, y)^ki(v, y)^l\begin{pmatrix}b\\k\end{pmatrix} uy^k\otimes vy^l,\label{deltaxyb}
\end{eqnarray}
\begin{eqnarray}
&&b\Delta(x^2y^b) =\{\Delta(x), xy^b\}+\{x, \Delta(xy^b)\}\nonumber\\
&=&\sum_{u, v\in\pi_1(S)^{\mathrm{Ab}}}C^x_{u, v}\{u\otimes v, xy^b\}\nonumber\\&&+\sum_{u, v\in\pi_1(S)^{\mathrm{Ab}}}C^x_{u, v}\sum_{k, l\geq 0, k+l=b}i(u, y)^ki(v, y)^l\begin{pmatrix}b\\k\end{pmatrix} \{x, uy^k\otimes vy^l\}.\label{deltax2y}
\end{eqnarray}
Define a lexicographic order on $\{x^ay^b\otimes x^cy^d\ |\ a, b, c, d\in\mathbb{Z}\}$ as follows.
$x^ay^b\otimes x^cy^d<x^{a'}y^{b'}\otimes x^{c'}y^{d'}$ if and only if
\begin{itemize}\item$a<a'\ \text{or}$
\item$a=a'\ \text{and}\ b<b',\ \text{or}$
\item$a=a'\ \text{and}\ b=b'\ \text{and}\ c<c',\ \text{or}$
\item$a=a'\ \text{and}\ b=b'\ \text{and}\ c=c'\ \text{and}\ d<d'$.\end{itemize}

Let $u_0\otimes v_0$ be the maximum element satisfying $C^x_{u_0, v_0}\neq 0$ in the order.
For $b, b'>0$ satisfying $b\neq b'$, the maximum term in
$i(xy^b, xy^{b'})\Delta(x^2y^{b+b'})=\{\Delta(xy^b), xy^{b'}\}+\{xy^b, \Delta(xy^{b'})\}$
is $xy^{b+b'}u_0\otimes v_0$.
The coefficient of $xy^{b+b'}u_0\otimes v_0$ is
\begin{eqnarray}
&&\frac{b'-b}{b+b'}C^x_{u_0, v_0}(i(u_0, xy^{b+b'})+i(u_0, y)^{b+b'}i(x, y^{b+b'}u_0))\nonumber\\
&=&C^x_{u_0, v_0}(i(u_0, y)^bi(y^bu_0, xy^{b'})+i(u_0, y)^{b'}i(xy^b, y^{b'}u_0)).\label{eq:maximal_term}
\end{eqnarray}
We can choose $\lambda\in\mathbb{Z}$ satisfying $\lambda> 1$ and $\lambda+i(y, u_0)\neq 0$.
If $|i(u_0, y)|>1$, substituting $b'=\lambda b$ into (\ref{eq:maximal_term}), for sufficiently large $b>0$,
\begin{eqnarray*}
&&\frac{b'-b}{b+b'}C^x_{u_0, v_0}(i(u_0, xy^{b+b'})+i(u_0, y)^{b+b'}i(x, y^{b+b'}u_0))\\&=& \mathcal{O}(b\cdot i(u_0, y)^{(1+\lambda)b}),\\
&&C^x_{u_0, v_0}(i(u_0, y)^bi(u_0y^b, xy^{b'})+i(u_0, y)^{b'}i(xy^b, y^{b'}u_0))\\
&=&\mathcal{O}(b\cdot i(u_0, y)^{\lambda b}),
\end{eqnarray*}
where $\mathcal{O}$ is the asymptotic notation, that is, $f(b)=\mathcal{O}(g(b))$ if and only if $\lim_{b\rightarrow +\infty}|\frac{f(b)}{g(b)}|$ exists and finite for functions $f, g$ on $\mathbb{R}$.
This contradicts (\ref{eq:maximal_term}).
If $i(u_0, y)=0$, then
$\frac{b'-b}{b+b'}C^x_{u_0, v_0}i(u_0, xy^{b+b'})=0$
for all $b, b'>0$ by (\ref{eq:maximal_term}). Hence $u_0=1$.
This contradicts the assumption $\Delta(x)\in W_1(1)'\otimes W_1(1)'$.
If $i(u_0, y)=-1$, when $b$ is odd and $b'$ is even,
\begin{eqnarray*}
&&\frac{b'-b}{b+b'}C^x_{u_0, v_0}(i(u_0, xy^{b+b'})+i(u_0, y)^{b+b'}i(x, y^{b+b'}u_0))\\
&=&\frac{b'-b}{b+b'}C^x_{u_0, v_0}\cdot 2(i(u_0, x)-(b+b')),\\
&&C^x_{u_0, v_0}(i(u_0, y)^bi(y^bu_0, xy^{b'})+i(u_0, y)^{b'}i(xy^b, y^{b'}u_0))\\
&=&C^x_{u_0, v_0}\cdot 2(-i(u_0, x)+b+b').
\end{eqnarray*}
This contradicts (\ref{eq:maximal_term}).
Hence we obtain $i(u_0, y)=1$. If $C^x_{u, v}\neq 0$, then
$\mathrm{deg}_x u\leq 1$.

Substituting $Z_1=y, Z_2=y^{-1}$ into (\ref{compatible_condition}), we have
$0=i(y, u)C^{y^{-1}}_{y^{-1}u, v}+i(y, v)C^{y^{-1}}_{u, y^{-1}v}$ from $\Delta(1)=\Delta(y)=0$.
So if $i(y, u)\neq 0$ or $i(y, v)\neq 0$, then $C^{y^{-1}}_{u, v}=0$.
Substituting $Z_1=y^{-1}, Z_2=xy$ into (\ref{compatible_condition}),
$C^x_{u, v}=i(u, xy)C^{y^{-1}}_{(xy)^{-1}u, v}+i(v, xy)C^{y^{-1}}_{u, (xy)^{-1}v}+i(y^{-1}, u)C^{xy}_{yu, v}+i(y^{-1}, v)C^{xy}_{u, yv}$.
If $i(y, u)\neq 0$ and $i(y, v)\neq 0$, then $C^{y^{-1}}_{(xy)^{-1}u, v}=0$ and $C^{y^{-1}}_{u, (xy)^{-1}v}=0$. From (\ref{compatible_condition}) and $\Delta(y)=0$, we have
\begin{eqnarray*}
C^x_{u, v}&=&i(y^{-1}, u)C^{xy}_{yu, v}+i(y^{-1}, v)C^{xy}_{u, yu}\\
&=&i(y^{-1}, u)(i(u, y)C^x_{u, v}+i(v, y)C^x_{yu, y^{-1}v})\\&&+i(y^{-1}, v)(i(u, y)C^x_{y^{-1}u, yv}+i(v, y)C^x_{u, v}).
\end{eqnarray*}
Hence we have
\begin{align*}
i(u, y)i(v, y)C^x_{y^{n-1}u, y^{-n+1}v}+(i(u, y)^2+i(v, y)^2-1)C^x_{y^n u, y^{-n}v}\\+i(u, y)i(v, y)C^x_{y^{n+1}u, y^{-n-1}v}=0
\end{align*}
for all $n\in\mathbb{Z}$.
There exists $N\in\mathbb{Z}$ such that $C^x_{y^nu, y^{-n}v}=0$ for all $n>N$.
By induction on $n\leq N$, since $i(u, y)i(v, y)\neq 0$ and $i(u, y)^2+i(v, y)^2-1>0$, we obtain $C^x_{y^nu, y^{-n}v}=0$ if $i(y, u)\neq 0$ and $i(y, v)\neq 0$.
So if $C^x_{u, v}\neq 0$ and $i(u, y)=1$, we can write $(u, v)=(xy^b, y^d)$ for some $b$ and $d\in\mathbb{Z}$.

We introduce
$\Delta':=\Delta-\sum_{u, v\in\pi_1(S)^{\mathrm{Ab}}, i(u, y)=1, i(x, u)\neq 0}\frac{C^x_{u, v}}{i(x, x^{-1}u)}d(x^{-1}u\otimes v)$,
and set
$\Delta'(Z)=\sum_{u, v\in\pi_1(S)^{\mathrm{Ab}}}C'^Z_{u, v}u\otimes v$
for every $Z\in\pi_1(S)^{\mathrm{Ab}}$.
The maximum element satisfying $C'^x_{u, v}\neq 0$ in the lexicographic order can be written as $x\otimes y^d$ for some $d\neq 0$.
For $b, b'>0$, 
\begin{align}\label{eq:coefxxy}
i(xy^b, xy^{b'})\Delta'(x^2y^{b+b'})=\{\Delta'(xy^b), xy^{b'}\}+\{xy^b, \Delta'(xy^{b'})\}.
\end{align}
By (\ref{deltax2y}), the coefficient of $x\otimes xy^{b+b'+d}$ in $(b+b')\Delta'(x^2y^{b+b'})$ is
\begin{align*}C'^x_{x, y^d}i(y^d, xy^{b+b'})+C'^x_{y^{-(b+b')}, xy^{b+b'+d}}i(y^{-(b+b')}, xy^{b+b'})\\
+C'^x_{xy^{-(b+b')}, y^{b+b'+d}}i(xy^{-(b+b')}, y)^{b+b'}i(x, y^{b+b'+d}).\end{align*}
If we take $b$ and $b'>0$ sufficiently large so that $-(b+b')<\mathrm{min}_{C'^x_{u, v}\neq 0}\mathrm{deg}_yu$, we obtain $C'^x_{y^{-(b+b')}, xy^{b+b'+d}}=0$ and $C'^x_{xy^{-(b+b')}, y^{b+b'+d}}=0$.
By (\ref{deltaxyb}), the coefficient of $x\otimes xy^{b+b'+d}$ in $\{\Delta'(xy^b), xy^{b'}\}$ is
$C'^{xy^b}_{y^{-b'}, xy^{b+b'+d}}i(y^{-b'}, xy^{b'})+C'^{xy^b}_{x, y^{b+d}}i(y^{b+d}, xy^{b'})$.
Since $b+d>d$, we obtain $C'^x_{x, y^{b+d}}=0$.
If we take $b'>0$ sufficiently large so that $-b'<\mathrm{min}_{C'^x_{u, v}\neq 0}\mathrm{deg}_yu$, we obtain $C'^x_{y^{-b'}, xy^{b+b'+d}}=0$.
Similarly the coefficient of $x\otimes xy^{b+b'+d}$ in $\{xy^b, \Delta'(xy^{b'})\}$ is $0$ if $-b<\mathrm{min}_{C'^x_{u, v}\neq 0}\mathrm{deg}_yu.$
Hence the coefficient of $x\otimes xy^{b+b'+d}$ in (\ref{eq:coefxxy}) is
$\frac{b'-b}{b+b'}C'^x_{x, y^d}i(y^d, xy^{b+b'})=0$
when
$-b, -b'<\mathrm{min}_{C'^x_{u, v}\neq 0}\mathrm{deg}_yu.$
This contradicts $d\neq 0$.
Consequently $\Delta'(x)=0$ and $\Delta'(y)=0$.

\end{proof}

\begin{lemm}\label{lem:xy}
For $\Delta\in Z^1(W_1(1), W_1(1)'\otimes W_1(1)')$, there exists $\alpha\in W_1(1)'\otimes W_1(1)'$ such that
$(\Delta-d\alpha)(y)=0$.
\end{lemm}

\begin{proof}
Set
$\Delta(Z)=\sum_{u, v\in\pi_1(S)^{\mathrm{Ab}}}C^Z_{u, v}u\otimes v$
for every $Z\in\pi_1(S)^{\mathrm{Ab}}$.
We have
\begin{align*}
\Delta(y)&=\sum_{i(y, u)\neq 0, i(y, v)=0}\frac{C^y_{u, v}}{i(y, u)}d(y^{-1}u\otimes v)(y)+\sum_{i(y, u)=0, i(y, v)\neq 0}\frac{C^y_{u, v}}{i(y, v)}d(u\otimes y^{-1}v)(y)\\
&+\sum_{i(y, u)\neq 0, i(y, v)\neq 0}C^y_{u, v}u\otimes v+\sum_{i(y, u)=0, i(y, v)=0}C^y_{u, v}u\otimes v.
\end{align*}
Replacing $\Delta$ with 
\[
\Delta-\sum_{i(y, u)\neq 0, i(y, v)=0}\frac{C^y_{u, v}}{i(y, u)}d(y^{-1}u\otimes v)-\sum_{i(y, u)=0, i(y, v)\neq 0}\frac{C^y_{u, v}}{i(y, v)}d(u\otimes y^{-1}v),
\]
we may assume $C^y_{u, v}=0$ if $i(y, u)\neq 0$ and $i(y, v)=0$, or $i(y, u)=0$ and $i(y, v)\neq 0$.

If $i(y, u_0)\neq 0$ and $i(y, v_0)\neq 0$, consider the sequence
$\{C^y_{y^{-n}u_0, y^nv_0}\}_{n\in\mathbb{Z}}$.
We only have to consider one representative $(u_0, v_0)$ of each orbit for $\mathbb{Z}$.
Let $N\in\mathbb{Z}$ be the minimum integer satisfies $C^y_{y^{-n}u_0, y^nv_0}=0$ for all $n>N$.
Replace $\Delta$ with $\overline\Delta:=\Delta-\frac{C^y_{y^{-N}u_0, y^Nv_0}}{i(y, v_0)}d(y^{-N}u_0\otimes y^{N-1}v_0)$, and set $\overline\Delta(Z)=\sum_{u, v\in\pi_1(S)^{\mathrm{Ab}}}\overline C^Z_{u, v}u\otimes v$ for all $Z\in\pi_1(S)^{\mathrm{Ab}}$.
Then $\overline C_{y^{-N}u_0, y^Nv_0}^y=0$ and $\overline C^y_{y^{-N+1}u_0, y^{N-1}v_0}=C^y_{y^{-N+1}u_0, y^{N-1}v_0}-\frac{i(y, u_0)}{i(y, v_0)}C^y_{y^{-N}u_0, y^Nv_0}$.
Repeat this operation $R+N$ times.
For a sufficiently large $R>0$, by replacing $\Delta$ with
\begin{align*}
\Delta-\sum_{n=-N}^{R-1}\frac{1}{i(y, v_0)}(C^y_{y^nu_0, y^{-n}v_0}-\frac{i(y, u_0)}{i(y, v_0)}C^y_{y^{n-1}u_0, y^{-n-1}v_0}+\cdots \\+\left(-\frac{i(y, u_0)}{i(y, v_0)}\right)^{n+N}C^y_{y^{-N}u_0, y^Nv_0})d(y^nu_0\otimes y^{-n-1}v_0),
\end{align*}
we can assume $C^y_{y^{n}u_0, y^{-n} v_0}=0$ for $n\neq R$.
We will show $C^y_{y^Ru_0, y^{-R}v_0}=0$.
Suppose $C^y_{y^Ru_0, y^{-R}v_0}\neq 0$.
Let us define a lexicographic order on $\{x^ay^b\otimes x^cy^d \ |\ a, b, c, d\in\mathbb{Z}\}$ as follows.
$x^ay^b\otimes x^cy^d <x^{a'}y^{b'}\otimes x^{c'}y^{d'}$ if and only if
\begin{itemize}
\item$b<b'\ \text{or}$
\item $b=b'\ \text{and}\ a<a',\ \text{or}$
\item $b=b'\ \text{and}\ a=a'\ \text{and}\ d<d',\ \text{or}$
\item $b=b'\ \text{and}\ a=a'\ \text{and}\ d=d'\ \text{and}\ c<c'.$\end{itemize}

Take $R$ sufficiently large so that $R>\max_{C^y_{u, v}\neq 0}\deg_y u-\deg_y u_0, \max_{C^x_{u, v}\neq 0}\deg_y u +1-\deg_y u_0$.
Then the maximum term in $\Delta(y)$ is $y^Ru_0\otimes y^{-R}v_0$ and that in
$\Delta(xy)=\{\Delta(x), y\}+\{x, \Delta(y)\}$
is $xy^Ru_0\otimes y^{-R}v_0$.
The coefficient of $xy^Ru_0\otimes y^{-R}v_0$ in $\Delta(xy)$ is
$i(x, y^Ru_0)C^y_{y^Ru_0, y^{-R}v_0}$.
By induction on $a>0$, the maximum term in
$\Delta(x^ay)=\{\Delta(x), x^{a-1}y\}+\{x, \Delta(x^{a-1}y)\}$
is $x^ay^Ru_0\otimes y^{-R}v_0$, and its coefficient is
$i(x, y^Ru_0)^a C^y_{y^Ru_0, y^{-R}v_0}$.
The maximum term in
$a\Delta(x^ay^2)=\{\Delta(x^ay), y\}+\{x^ay, \Delta(y)\}$
is $x^ay^{R+1}u_0\otimes y^{-R}v_0$ and its coefficient is
$(i(x^ay^Ru_0, y)i(x, y^Ru_0)^a+i(x^ay, y^Ru_0))C^y_{y^Ru_0, y^{-R}v_0}$.
For $a, a'>0$, the coefficient of $x^{a+a'}y^{R+1}u_0\otimes y^{-R}v_0$ in
$i(x^ay, x^{a'}y)\Delta(x^{a+a'}y^2)=\{\Delta(x^ay), x^{a'}y\}+\{x^ay, \Delta(x^{a'}y)\}$
is
\begin{eqnarray}
&&\frac{a-a'}{a+a'}(i(x^{a+a'}y^Ru_0, y)i(x, y^Ru_0)^{a+a'}+i(x^{a+a'}y, y^Ru_0))C^y_{y^Ru_0, y^{-R}v_0}\label{eq:maximal_term_y}\\
&=&i(x^ay^Ru_0 x^{a'}y)i(x, y^Ru_0)^aC^y_{y^Ru_0, y^{-R}v_0}+i(x^ay, x^{a'}y^Ru_0)i(x, y^Ru_0)^{a'}C^y_{y^Ru_0, y^{-R}v_0}.\nonumber
\end{eqnarray}
We can choose $\lambda\in\mathbb{Z}$ satisfying $\lambda>1$ and $R+i(x, u_0)-\lambda\neq 0$.
Since we can assume $|i(x, y^Ru_0)|>1$, substituting $a'=\lambda a$, for sufficiently large $a>0$,
\begin{eqnarray*}
&&\frac{a-a'}{a+a'}(i(x^{a+a'}y^Ru_0, y)i(x, y^Ru_0)^a+i(x^{a+a'}y, y^Ru_0))C^y_{y^Ru_0, y^{-R}v_0}\\&=&\mathcal{O}(a\cdot i(x, y^Ru_0)^{(1+\lambda)a}),\\
&&i(x^ay^Ru_0, x^{a'}y)i(x, y^Ru_0)^aC^y_{y^Ru_0, y^{-R}v}+i(x^ay, x^{a'}y^Ru_0)i(x, y^Ru_0)^{a'}C^y_{y^Ru_0, y^{-R}v_0}\\&=&\mathcal{O}(a\cdot i(x, y^Ru_0)^{\lambda a}).
\end{eqnarray*}
This contradicts (\ref{eq:maximal_term_y}).
Hence $C^y_{y^Ru_0, y^{-R}v_0}=0$.
So we can assume
$\Delta(y)=\sum_{k, l\in\mathbb{Z}, k, l\neq 0}C^y_{y^k, y^l}y^k\otimes y^l$.

Let $y^k\otimes y^l$ be the maximum term satisfying $C^y_{y^k, y^l}\neq 0$.
If we replace $\Delta$ with 
$\overline\Delta:=\Delta-\sum_{s, t\in\mathbb{Z}, s\neq 0}\frac{C^x_{xy^s, y^t}}{i(x, y^s)}d(y^s\otimes y^t)$, then $\overline C^x_{xy^s, y^t}=0$ for all $s\neq 0$ and $t\in\mathbb{Z}$, where $\overline\Delta(x)=\sum_{u, v\in\pi_1(S)^{\mathrm{Ab}}}\overline C^x_{u, v}u\otimes v$.
Since $\overline\Delta(y)=\Delta(y)$, we can assume $C^x_{xy^s, y^t}=0$ for all $s\neq 0$ and $t\in\mathbb{Z}$.
If $k\neq 1$, then the coefficient of $xy^k\otimes y^l$ in
$\Delta(xy)=\{\Delta(x), y\}+\{x, \Delta(y)\}$
is $kC^y_{y^k, y^l}$.
By induction on $b>0$, the coefficient of $xy^{k+b-1}\otimes y^l$ in
$\Delta(xy^b)=\{\Delta(xy^{b-1}), y\}+\{xy^{b-1}, \Delta(y)\}$
is $bkC^y_{y^k, y^l}$.
Substituting $Z_1=y^{-(k-1)}, Z_2=xy^{k+b-1}$ into (\ref{compatible_condition}), the coefficient of $xy^{k+b-1}\otimes y^l$ in
$(k-1)\Delta(xy^b)=\{\Delta(y^{-(k-1)}), xy^{k+b-1}\}+\{y^{-(k-1)}, \Delta(xy^{k+b-1})\}$
is 
\begin{eqnarray*}
&&(k-1)\cdot bk C^y_{y^k, y^l}\\
&=&C^{y^{-(k-1)}}_{1, y^l}i(1, xy^{k+b-1})+C^{y^{-(k-1)}}_{xy^{k+b-1}, x^{-1}y^{l-k-b+1}}i(x^{-1}y^{l-k-b+1}, xy^{k+b-1})\\&&\ +i(y^{-(k-1)}, xy^{2k+b-2})\cdot (k+b-1)k C^y_{y^k, y^l}+i(y^{-(k-1)}, y^{k+l-1})C^{xy^{k+b-1}}_{xy^{k+b-1}, y^{k+l-1}}\\
&=& (k-1)\cdot (k+b-1)kC^y_{y^k, y^l}
\end{eqnarray*}
for $b\gg 0$.
Since $k\neq 0$ and $C^y_{y^k, y^l}\neq 0$, this contradicts the assumption $k\neq 1$.
Therefore $k=1$.
Hence the maximum term in $\Delta(xy^b)$ is $xy^b\otimes y^l$ for $b>0$.
Substituting $Z_1=x, Z_2=xy^b$ into (\ref{compatible_condition}), the coefficient of $xy^b\otimes xy^l$ in
$b\Delta(x^2y^b)=\{\Delta(x), xy^b\}+\{x, \Delta(xy^b)\}$
is
\[
i(1, xy^b)C^x_{1, xy^l}+i(y^{l-b}, xy^b)C^x_{xy^b, y^{l-b}}+i(x, y^b)C^{xy^b}_{y^b, xy^l}+i(x, y^l)C^{xy^b}_{xy^b, y^l}.
\]
Since $C^x_{xy^s, y^t}=0$ for all $s, t\neq 0$, we have $C^x_{xy^b, y^{l-b}}=0$.
We compute the value $C^{xy^b}_{y^b, xy^l}$.
Take $b>\mathrm{max}_{C^x_{u, v}\neq 0}\{\mathrm{deg}_yu\}, \mathrm{max}_{C^y_{u, v}\neq 0}\{\mathrm{deg}_yu\}$.
From the equation $\Delta(xy^b)=\{\Delta(xy^{b-1}), y\}+\{xy^{b-1}, \Delta(y)\}$, we have
\begin{align*}
&C^{xy^b}_{y^b, xy^l}\\
&= i(y^b, y)C^{xy^{b-1}}_{y^{b-1}, xy^l}+i(xy^l, y)C^{xy^{b-1}}_{y^b, xy^{l-1}}+i(xy^{b-1}, y^b)C^y_{x^{-1}y, xy^l}+i(xy^{b-1}, xy^l)C^y_{y^b, y^{l-b+1}}\\
&=i(xy^l, y)C^{xy^{b-1}}_{y^b, xy^{l-1}}+i(xy^{b-1}, y^b)C^y_{x^{-1}y, xy^l}.
\end{align*}
Since $\Delta(y)=\sum_{k, l\in\mathbb{Z}}C^y_{y^k, y^l}y^k\otimes y^l$, we obtain
\begin{eqnarray}
C^{xy^b}_{y^b, xy^l}
=C^{xy^{b-1}}_{y^b, xy^{l-1}}
=\cdots =C^x_{y^b, xy^{l-b}}=0.\label{eq:yxy0}
\end{eqnarray}
Hence the coefficient of the term $xy^b\otimes xy^l$ in $\Delta(x^2y^b)$ is
$\frac{l}{b}C^{xy^b}_{xy^b, y^l}=\frac{l}{b}(C^x_{x, y^l}+bC^y_{y, y^l})$.
Substituting $Z_1=xy^b, Z_2=xy^{b'}$ into (\ref{compatible_condition}), 
$(b'-b)\Delta(x^2y^{b+b'})=\{\Delta(xy^b), xy^{b'}\}+\{xy^b, \Delta(xy^{b'})\}$.
For $b, b'>0$, the coefficient of $xy^{b+b'}\otimes xy^l$ is
\begin{eqnarray*}
\frac{b'-b}{b+b'}l(C^x_{x, y^l}+bC^y_{y, y^l})&=&i(y^b, xy^{b'})C^{xy^b}_{y^b, xy^l}+i(y^{l-b'}, xy^{b'})C^{xy^b}_{xy^{b+b'}, y^{l-b'}}\\&&+i(xy^b, y^{b'})C^{xy^{b'}}_{y^{b', xy^l}}+i(xy^b, y^{l-b})C^{xy^{b'}}_{xy^{b+b'}, y^{l-b}}.
\end{eqnarray*}
Since $b, b' < b+b'$, we have $C^{xy^b}_{xy^{b+b'}, y^{l-b'}}=0$ and $C^{xy^{b'}}_{xy^{b+b'}, y^{l-b}}=0$.
If we take sufficiently large $b$ and $b'$, then from a similar argument to (\ref{eq:yxy0}), we obtain $C^{xy^b}_{y^b, xy^l}=0$ and $C^{xy^{b'}}_{y^{b'}, xy^l}=0$.
This contradicts $l\neq 0$ and $C^y_{y, y^l}\neq 0$.
Therefore $\Delta(y)=0$.
This proves Lemma \ref{lem:xy}.
\end{proof}

By Lemma \ref{lem:y0} and Lemma \ref{lem:xy}, we obtain Theorem \ref{thm:iso_g} for $g=1$. Now we prove Theorem \ref{thm:iso_g} for any $g$.

\begin{lemm}\label{lem:gdeltay}
Let $i\in\{1, \cdots, g\}$ and $\Delta\in Z^1(W_1(g), W_1(g)'\otimes W_1(g)')$ satisfy $\Delta(x_{j})=0$ and $\Delta(y_{j})=0$ for all $j>i$.
Then there exists some element $\alpha\in W_1(g)'\otimes W_1(g)'$ such that the difference $\Delta'=\Delta-d\alpha$ satisfies $\Delta'(y_i)=0$, $\Delta'(x_j)=0$ and $\Delta'(y_j)=0$ for all $j>i$.
\end{lemm}
\begin{proof}
For $j>i$ and $z_j\in\{x_j, y_j\}$, substituting $Z_1=y_i, Z_2=z_j$ into (\ref{compatible_condition}), we have
$0=i(u, z_j)C^{y_i}_{z_j^{-1}u, v}+i(v, z_j)C^{y_i}_{u, z_j^{-1}v}$ from $\Delta(z_j)=0$.
Since $C^{y_i}_{z_j^{n-1}u, z_j^n v}=0$ for sufficiently large $n$, we have $C^{y_i}_{u, v}=0$ if $i(u, z_j)\neq 0$ or $i(v, z_j)\neq 0$.
So we may write
$\Delta(y_i)=\sum_{u, v\in \langle x_1, y_1, \cdots, x_i, y_i\rangle}C^{y_i}_{u, v}u\otimes v$,
where $\langle x_1, y_1, \cdots, x_i, y_i\rangle$ is the subgroup of $\pi_1(S)^{\mathrm{Ab}}$ which is generated by $x_1, y_1, \cdots, x_i, y_i$.
By a similar argument to the proof of Lemma \ref{lem:xy}, we can assume
\[
\Delta(y_i)=\sum_{u, v\neq 1\in\langle x_1, y_1, \cdots, x_i, y_i\rangle, i(y_i, u)=0\ \text{and}\ i(y_i, v)=0}C^{y_i}_{u, v}u\otimes v.
\]
In other words, let us define a set $T:=\{(y_i^ku_1, v_1)\ |\ k\in\mathbb{Z}, u_1, v_1\in\langle x_1, y_1, \cdots, x_{i-1}, y_{i-1}, x_i\rangle\}$ and take $R\in\mathbb{Z}$ sufficiently large so that \[R>\max_{C^{y_i}_{u, v}\neq 0}\{\deg_y u, 0\}-\min_{C^{y_i}_{u, v}\neq 0}\deg_y u, \max_{C^{x_i}_{u, v}\neq 0}\deg_y u +1-\min_{C^{y_i}_{u, v}\neq 0}\deg_y u.\]
Replacing $\Delta$ with 
\begin{align*}
\Delta-\sum_{i(y_i, u)\neq 0, i(y_i, v)=0}\frac{C^{y_i}_{u, v}}{i(y_i, u)}d(y_i^{-1}u\otimes v)-\sum_{i(y_i, u)=0, i(y_i, v)\neq 0}\frac{C^{y_i}_{u, v}}{i(y_i, v)}d(u\otimes y_i^{-1}v)\\
-\sum_{i(y_i, u)\neq 0, i(y_i, v)\neq 0, (u, v)\in T}\sum_{n=-N}^{R-1}\frac{1}{i(y_i, v)}(C^y_{y_i^nu, y_i^{-n}v}-\frac{i(y_i, u)}{i(y_i, v)}C^{y_i}_{y_i^{n-1}u, y_i^{-n-1}v}+\cdots \\
+ \left(-\frac{i(y_i, u)}{i(y_i, v)}\right)^{n+N}C^{y_i}_{y_i^{-N}u, y_i^Nv})d(y_i^nu\otimes y_i^{-n-1}v),
\end{align*}
we have $C^{y_i}_{u, v}=0$ when $i(y_i, u)\neq 0$ or $i(y_i, v)\neq 0$.
We introduce a lexicographic order on $\{x_1^{a_1}y_1^{b_1}\cdots x_g^{a_g}y_g^{b_g}\otimes x_1^{c_1}y_1^{d_1}\cdots x_g^{c_g}y_g^{d_g}|\ a_i, b_i, c_i , d_i\in\mathbb{Z}\}$ as follows.
$x_1^{a_g}y_1^{b_g}\cdots x_g^{a_1}y_g^{b_1}\otimes x_1^{a_{2g}}y_1^{b_{2g}}\cdots x_g^{a_{g+1}}y_g^{b_{g+1}}<x_1^{a'_g}y_1^{b'_g}\cdots x_g^{a'_1}y_g^{b'_1}\otimes x_1^{a'_{2g}}y_1^{b'_{2g}}\cdots x_g^{a'_{g+1}}y_g^{b'_{g+1}}$ if and only if there exists $j\in\{1, \cdots, 2g\}$ such that
$b_k=b'_k$ and $a_k=a'_k$ for all $k<j$ and 
\begin{itemize}\item $b_j<b'_j\ \text{or}$ \item $b_j=b'_j\ \text{and}\ a_j<a'_j$.\end{itemize}

Substituting $Z_1=x_iy_i, Z_2=x_i^{-1}$ into (\ref{compatible_condition}), we obtain
\begin{eqnarray*}
C^{y_i}_{u, v}&=&i(u, x_i^{-1})C^{x_iy_i}_{x_iu, v}+i(v, x_i^{-1})C^{x_iy_i}_{u, x_iv}\\&&+i(x_iy_i, u)C^{x_i^{-1}}_{(x_iy_i)^{-1}u, v}+i(x_iy_i, v)C^{x_i^{-1}}_{u, (x_iy_i)^{-1}v}.
\end{eqnarray*}
Hence we have $i(x_i, u)\neq 0$ or $i(x_i, v)\neq 0$ if $C^{y_i}_{u, v}\neq 0$.

Let $u_0\otimes v_0$ be the maximum term satisfying $C^{y_i}_{u_0, v_0}\neq 0$.
Then we have $i(x_i, u_0)\neq 0$ or $i(x_i, v_0)\neq 0$.
If $i(x_i, u_0)\neq 0$ and $i(x_i, v_0)\neq 0$, then a similar argument to the proof of Lemma \ref{lem:xy} holds.
That is, we can write $u_0\otimes v_0=y_i^ku_1\otimes y_i^{l}v_1$ for some $u_1, v_1\in\langle x_1, y_1, \cdots, x_{i-1}, y_{i-1}\rangle$ and some $k, l\neq 0$.
Considering the coefficient of $x_iy_i^{k+b-1}u_1\otimes x_iy_i^lv_1$ in $\Delta(x_i^2y_i^b)$ for a sufficiently large $b>0$, we obtain $l=0$.
This contradicts the assumption $l\neq 0$.

If $i(x_i, u_0)=0$ and $i(x_i, v_0)\neq 0$, we can write $u_0\otimes v_0=u_1\otimes y_i^lv_1$ where $u_1, v_1\in\langle x_1, y_1, \cdots, x_{i-1}, y_{i-1}\rangle\subset\pi_1(S)^{\mathrm{Ab}}$ and $l\neq 0$.
From $u_1\neq 1$, there exists $j<i$ and $z_j\in\{x_j, y_j\}$ such that $i(z_j, u_1)\neq 0$.
Substituting $Z_1=z_j$ and $Z_2=y_i$ into (\ref{compatible_condition}), we have
\[0=i(u, y_i)C^{z_j}_{y_i^{-1}u, v}+i(v, y_i)C^{z_j}_{u, y_i^{-1}v}+i(z_j, u)C^{y_i}_{z_j^{-1}u, v}+i(z_j, v)C^{y_i}_{u, z_j^{-1}v}.\]
Substituting $z_j^{-1}u=u_1$ and $v=y_i^lv_1$ into the above equation, we obtain
\[0=i(z_j, u_1)C^{y_i}_{u_1, y_i^lv_1}+i(z_j, v_1)C^{y_i}_{z_ju_1, z_j^{-1}y_i^lv_1}.\]
Since $i(z_j, u_1)\neq 0$, we have $C^{y_i}_{u_1, y_i^lv_1}=0$.
This contradicts the assumption $C^{y_i}_{u_1, y_i^lv_1}\neq 0$.
Similar argument holds for the case $i(x_i, u_0)\neq 0$ and $i(x_i, v_0)=0$.
Hence $\Delta(y_i)=0$.
This proves Lemma \ref{lem:gdeltay}.

\end{proof}

\begin{lemm}\label{lem:12}
Let $i\in\{1, \cdots, g\}$ and $\Delta\in Z^1(W_1(g), W_1(g)'\otimes W_1(g)')$ satisfy $\Delta(y_i)=0$, $\Delta(x_{j})=0$ and $\Delta(y_{j})=0$ for all $j>i$.
Then there exists $\alpha\in W_1(g)'\otimes W_1(g)'$ such that the difference $\Delta'=\Delta-d\alpha$ satisfies
$\Delta'(x_{j})=0$ and $\Delta'(y_{j})=0$ for all $j\geq i$.
\end{lemm}
\begin{proof}
For $j>i$ and $z_j\in\{x_j, y_j\}$, substituting $Z_1=x_i, Z_2=z_j$ into (\ref{compatible_condition}), we have
$0=i(u, z_j)C^{x_i}_{z_j^{-1}u, v}+i(v, z_j)C^{x_i}_{u, z_j^{-1}v}$ from $\Delta(z_j)=0$.
Therefore $C^{x_i}_{u, v}=0$ if $i(u, z_j)\neq 0$ or $i(v, z_j)\neq 0$.
So we may write
$\Delta(x_i)=\sum_{u, v\in\langle x_1, y_1, \cdots, x_i, y_i\rangle}C^{x_i}_{u, v}u\otimes v$.

Define $\Delta'$ by
\begin{eqnarray*}
\Delta'&=&\Delta-(\sum_{k, l\in\mathbb{Z}, k\neq 0, u_1, v_1\in\langle x_1, y_1, \cdots, x_{i-1}, y_{i-1}\rangle}\frac{C^{x_i}_{x_iy_i^ku_1, y_i^lv_1}}{i(x_i, y_i^ku_1)}d(y_i^ku_1\otimes y_i^lv_1)\\&&+\sum_{l\neq 0\in\mathbb{Z}, u_1, v_1\in\langle x_1, y_1, \cdots, x_{i-1}, y_{i-1}\rangle}\frac{C^{x_i}_{u_1, x_iy_i^lv_1}}{i(x_i, y_i^lv_1)}d(u_1\otimes y_i^lv_1)).
\end{eqnarray*}
Set
$\Delta'(Z)=\sum_{u, v\in\pi_1(S)^{\mathrm{Ab}}}C'^Z_{u, v}u\otimes v$
for every $Z\in\pi_1(S)^{\mathrm{Ab}}$.
Let $u_0\otimes v_0$ be the maximum element satisfying $C'^{x_i}_{u_0, v_0}\neq 0$ in a lexicographic order, that is, $x_1^{a_g}y_1^{b_g}\cdots x_g^{a_1}y_g^{b_1}\otimes x_1^{a_{2g}}y_1^{b_{2g}}\cdots x_g^{a_{g+1}}y_g^{b_{g+1}}<x_1^{a'_g}y_1^{b'_g}\cdots x_g^{a'_1}y_g^{b'_1}\otimes x_1^{a'_{2g}}y_1^{b'_{2g}}\cdots x_g^{a'_{g+1}}y_g^{b'_{g+1}}$ if and only if there exists $j\in\{1, \cdots, 2g\}$ such that $a_k=a'_k$ and $b_k=b'_k$ for all $k<j$ and
\begin{itemize}\item $a_j<a'_j\ \text{or}$\item $a_j=a'_j\ \text{and}\ b_j<b'_j.$\end{itemize}
For $b\neq b'$, considering the maximum term in $i(x_iy_i^b, x_iy_i^{b'})\Delta'(x_i^2y_i^{b+b'})=\{\Delta'(x_iy_i^b), x_iy_i^{b'}\}+\{x_iy_i^b, \Delta'(x_iy_i^{b'})\}$, we have
\begin{eqnarray}
&&\frac{b'-b}{b+b'}C'^{x_i}_{u_0, v_0}(i(u_0, x_iy_i^{b+b'})+i(u_0, y_i)^{b+b'}i(x_i, y_i^{b+b'}u_0))\nonumber\\
&=&C'^{x_i}_{u_0, v_0}(i(u_0, y_i)^bi(y_i^bu_0, x_iy_i^{b'})+i(u_0, y_i)^{b'}i(x_iy_i^b, y_i^{b'}u_0)).\label{eq:maximal_term_g}
\end{eqnarray}
By a similar argument to a proof of Lemma \ref{lem:y0}, we obtain $i(u_0, y_i)=0$ or $i(u_0, y_i)=1$.

If $i(u_0, y_i)=0$, we have $i(u_0, x_i)=0$ by (\ref{eq:maximal_term_g}).
So we can write $u_0\otimes v_0=u_1\otimes x_i^sy_i^tv_1$, where $u_1, v_1\in\langle x_1, y_1, \cdots, x_{i-1}, y_{i-1}\rangle\subset\pi_1(S)^{\mathrm{Ab}}$, $s, t\in\mathbb{Z}$.

Considering the coefficient of $u_1\otimes x_i^{s+1}y_i^{t+b+b'}v_1$ in
$i(x_iy_i^{b}, x_iy_i^{b'})\Delta'(x_i^2y_i^{b+b'})=\{\Delta'(x_iy_i^{b}), x_iy_i^{b'}\}+\{x_iy_i^{b}, \Delta'(x_iy_i^{b'})\}$,
we obtain $s=1$.
Namely, for $b>0$, the maximum term in
\begin{eqnarray*}
\Delta'(x_iy_i^{b})&=&\{\cdots\{\{\Delta'(x_i), \underbrace{y_i\}, y_i\}, \cdots, y_i}_{b}\}\\
&=&\sum_{u, v\in\pi_1(S)^{\mathrm{Ab}}}C'^{x_i}_{u, v}\sum_{k, l\geq 0, k+l=b}i(u, y_i)^ki(v, y_i)^l\begin{pmatrix}b\\ k\end{pmatrix}y_i^ku\otimes y_i^lv
\end{eqnarray*}
is $u_1\otimes x_i^sy_i^{t+b}v_1$ and its coefficient is
$s^{b}C'^{x_i}_{u_1, x_i^sy_i^tv_1}$.
The maximum term in
$b\Delta'(x_i^2y_i^{b})=\{\Delta'(x_i), x_iy_i^{b}\}+\{x_i, \Delta'(x_iy_i^{b})\}$
is $u_1\otimes x_i^{s+1}y_i^tv_1$ and its coefficient is
$(i(x_i^sy_i^tv_1, x_iy_i^{b})+s^{b}i(x_i, y_i^t))C'^{x_i}_{u_1, x_1^sy_i^tv_1}$.
Hence the coefficient of $u_1\otimes x_i^{s+1}y_i^{t+b+b'}v_1$ in
$i(x_iy_i^{b}, x_iy_i^{b'})\Delta'(x_i^2y_i^{b+b'})=\{\Delta'(x_iy_i^{b}), x_iy_i^{b'}\}+\{x_iy_i^{b}, \Delta'(x_iy_i^{b'})\}$
is
\[
\frac{b'-b}{b+b'}C'^{x_i}_{u_1, x_i^sy_i^tv_1}(i(x_i^sy_i^t, x_iy_i^{b+b'})+s^{b+b'}t)
=C'^{x_i}_{u_1, x_i^sy_i^tv_1}(s^{b}i(x_i^sy_i^t, x_iy_i^{b'})+s^{b'}i(x_iy_i^{b}, x_i^sy_i^t))
\]
for $b, b'>0$.
Therefore $|s|<1$.
If $s=0$, we have $t=0$ from the above equation.
This contradicts the assumption $C'^{x_i}_{u_0, v_0}\neq 0$.
If $s=-1$, when $b$ is odd and $b'$ is even,
\begin{eqnarray*}
(i(x_i^sy_i^t, x_iy_i^{b+b'})+s^{b+b'}t)&=&-b-b'-2t,\\
s^{b}i(x_i^sy_i^t, x_iy_i^{b'})+s^{b'}i(x_iy_i^{b}, x_i^sy_i^t)&=&b+b'+2t.
\end{eqnarray*}
This contradicts the above equation.
Hence $s=1$.
Since $C'^{x_i}_{u_1, x_iy_i^tv_1}\neq 0$, we have $t=0$ by the definition of $\Delta'$.
Since $u_1\neq 1$, there exist $j<i$ and $z_j\in\{x_j, y_j\}$ such that $i(z_j, u_1)\neq 0$.
Substituting $Z_1=z_j$, $Z_2=x_i$, $u=z_ju_1$ and $v=x_iv_1$ into (\ref{compatible_condition}), we have
\begin{align*}
0&=i(z_ju_1, x_i)C'^{z_j}_{x_i^{-1}z_ju_1, x_iv_1}+i(x_iv_1, x_i)C'^{z_j}_{z_ju_1, v_1}+i(z_j, z_ju_1)C'^{x_i}_{u_1, x_iv_1}+i(z_j, x_iv_1)C'^{x_i}_{z_ju_1, z_j^{-1}x_iv_1}\\
&=i(z_j, u_1)C'^{x_i}_{u_1, x_iv_1}+i(z_j, v_1)C'^{x_i}_{z_ju_1, z_j^{-1}x_iv_1}.
\end{align*}
Since $i(z_j, u_1)\neq 0$, we obtain $C'^{x_i}_{u_1, x_iv_1}=0$. This contradicts the assumption $C'^{x_i}_{u_1, x_iv_1}=C'^{x_i}_{u_0, v_0}\neq 0$.

If $i(u_0, y_i)=1$, we can assume
$u_0\otimes v_0=x_iy_i^lu_1\otimes x_i^sy_i^tv_1$
for some $u_1, v_1\in\langle x_1, y_1, \cdots, x_{i-1}, y_{i-1}\rangle\subset\pi_1(S)^{\mathrm{Ab}}$ and $l, s, t\in\mathbb{Z}$. By a similar argument to the proof of Lemma \ref{lem:y0},
we have $C'^{x_i}_{u, v}=0$ if $i(y_i, u)\neq 0$ and $i(y_i, v)\neq 0$.
Hence $s=0$. By the definition of $\Delta'$, we have $l=0$.
If $t=0$, there exist $j<i$ and $z_j\in\{x_j, y_j\}$ such that $i(z_j, v_1)\neq 0$.
Substituting $Z_1=z_j$, $Z_2=x_i$, $u=x_iu_1$ and $v=z_jv_1$ into (\ref{compatible_condition}), we have
\begin{align*}
0&=i(x_iu_1, x_i)C'^{z_j}_{u_1, z_jv_1}+i(z_jv_1 x_i)C'^{z_j}_{x_iu_1, x_i^{-1}z_jv_1}+i(z_j, x_iu_1)C'^{x_i}_{z_j^{-1}x_iu_1, z_jv_1}+i(z_j, z_jv_1)C'^{x_i}_{x_iu_1, v_1}\\
&=i(z_j, u_1)C'^{x_i}_{z_j^{-1}x_iu_1, z_jv_1}+i(z_j, v_1)C'^{x_i}_{x_iu_1, v_1}.
\end{align*}
Since $i(z_j, v_1)\neq 0$, we obtain $C'^{x_i}_{x_iu_1, v_1}=0$. Therefore $t\neq 0$.

For a sufficiently large $b>0$, the coefficient of $x_iu_1\otimes x_iy_i^{t+b}v_1$ in the right-hand side of
$b\Delta'(x_i^2y_i^{b})=\{\Delta'(x_i), x_iy_i^{b}\}+\{x_i, \Delta'(x_iy_i^{b})\}$
is 
\begin{align*}
C'^{x_i}_{y_i^{-b}u_1, x_iy_i^{t+b}v_1}i(y_i^{-b}u_1, x_iy_i^b)+C'^{x_i}_{x_iu_1, y_i^tv_1}i(y_i^tv_1, x_iy_i^{b})+C'^{x_iy_i^b}_{u_1, x_iy_i^{t+b}v_1}i(x_i, u_1)\\+C'^{x_iy_i^b}_{x_iu_1, y_i^{t+b}v_1}i(x_i, y_i^{t+b}v_1)=C'^{x_i}_{x_iu_1, y_i^tv_1}i(y_i^tv_1, x_iy_i^{b}).
\end{align*}
For $b, b'>0$, the coefficient of $x_iu_1\otimes x_iy_i^{t+b+b'}v_1$ in $\{\Delta'(x_iy_i^{b}), x_iy_i^{b'}\}$ is 
\[
C'^{x_iy_i^{b}}_{y_i^{-b'}u_1, x_iy_i^{t+b+b'}v_1}i(y_i^{-b'}u_1, x_i y_i^{b'})+C'^{x_iy_i^b}_{x_iu_1, y_i^{t+b}v_1}i(y_i^{t+b}v_1, x_iy_i^{b'}).
\]
If $-b'<\min_{C'^x_{u, v}\neq 0}\deg_y u, \min_{C'^y_{u, v}\neq 0}\deg_y u$, we have $C'^{x_iy_i^b}_{y_i^{-b'}u_1, _iy_i^{t+b+b'}v_1}=0$.
Since $u_0\otimes v_0=x_iu_1\otimes y_i^t v_1$ is the maximal $u\otimes v$ satisfying $C'^{x_i}_{u, v}$ 
and $C'^{x_i}_{x_iy_i^pu_1, y_i^qv_1}=0$ for all $p\neq 0$, we have $C'^{x_iy_i^b}_{x_iu_1, y_i^{t+b}v_1}=0$.
Therefore, for $b, b'\gg 0$, the coefficient of $x_iu_1\otimes x_iy_i^{t+b+b'}v_1$ in 
$i(x_iy_i^{b}, x_iy_i^{b'})\Delta'(x_i^2y_i^{b+b'})=\{\Delta'(x_iy_i^{b}), x_iy_i^{b'}\}+\{x_iy_i^{b}, \Delta'(x_iy_i^{b'})\}$
is
$\frac{b'-b}{b+b'}C'^{x_i}_{x_iu_1, y_i^tv_1}i(y_i^tv_1, x_iy_i^{b+b'})=0$.
This contradicts $t\neq 0$.
Consequently $\Delta'(x_i)=0$.
This proves Lemma \ref{lem:12}.
\end{proof}
\end{subsection}
\end{section}

\begin{section}{Appendix: The Turaev cobracket represents a non-trivial cohomology class}
In this section we show the Turaev cobracket represents a non-trivial cohomology class in $H^1(\mathbb{Z}[\hat\pi]_0, \mathbb{Z}[\hat\pi]_0\wedge \mathbb{Z}[\hat\pi]_0)$.
\begin{prop}\label{prop:turaev_cobracket}
If $S$ is a closed oriented surface of genus $g\geq 2$, then
$[\Delta_T]\neq 0\in H^1(\mathbb{Z}[\hat\pi]_0, \mathbb{Z}[\hat\pi]_0\wedge\mathbb{Z}[\hat\pi]_0)$, where $\Delta_T$ is the Turaev cobracket on $\mathbb{Z}[\hat\pi]_0$.
\end{prop}
\begin{proof}
Assume $\Delta_T$ is a coboundary, that is, 
$\Delta_T=\sum_iC_i d(\alpha_i\wedge\beta_i)$, for some $\alpha_i, \beta_i\in\hat\pi$ and $C_i\in\mathbb{Z}$.
Let us define $\gamma:=[x_1x_2^{-1}]\in\hat\pi$ as in Figure \ref{fig:example_x1x2}.
\begin{figure}[htbp]
  \begin{center}
    \includegraphics[width=7cm,clip]{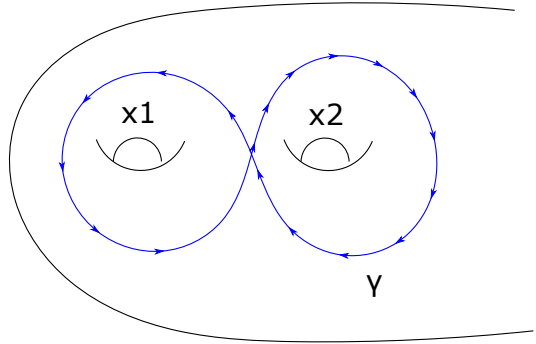}
    \caption{$\gamma=x_1x_2^{-1}$.}
    \label{fig:example_x1x2}
  \end{center}
\end{figure}
We have
\begin{align}
&(\sum_i C_id(\alpha_i\wedge\beta_i))(\gamma)=\sum_iC_i([\gamma, \alpha_i]\wedge\beta_i+\alpha_i\wedge [\gamma, \beta_i])\nonumber\\
&=\sum_iC_i(\sum_{p\in\gamma\cap\alpha_i}\epsilon(p; \gamma, \alpha_i)\gamma_p(\alpha_i)_p\wedge\beta_i+\sum_{w\in\gamma\cap\beta_i}\epsilon(p; \gamma, \beta_i)\alpha_i\wedge\gamma_q(\beta_i)_q),\label{eq:d}
\end{align}
where $[\ ,\ ]$ is the Goldman bracket.
From a natural surjection $\hat\pi\rightarrow\pi_1(S)^{\mathrm{Ab}}, \alpha\mapsto\overline\alpha$, we have a Lie algebra homomorphism
$p:\mathbb{Z}[\hat\pi]_0\rightarrow W_1(g)$.
Applying $p\wedge p$ for (\ref{eq:d}), we have
\begin{align*}
p\wedge p((\sum_i C_id(\alpha_i\wedge\beta_i))(\gamma))
=\sum_iC_i(i(\gamma, \alpha_i)\overline{\gamma\alpha_i}\wedge\overline\beta_i+i(\gamma, \beta_i)\overline \alpha_i\wedge\overline{\gamma\beta_i}).
\end{align*}
If $\overline {\gamma\alpha_i}=\overline x_1$ or $\overline x_2^{-1}$, then we have $i(\gamma, \alpha_i)=i(x_1x_2^{-1}, \gamma\alpha_i)=0$.
Therefore the coefficient of $x_1\wedge x_2^{-1}$ is $0$.
This contradicts
$p\wedge p(\Delta_T(\gamma))=\overline x_1\wedge\overline x_2^{-1}$.
Therefore the Turaev cobracket $\Delta_T$ represents a non-trivial cohomology class.
\end{proof}

\end{section}

\end{document}